\documentclass[11pt]{article}

\usepackage[a4paper, left=3.2cm,top=3.2cm]{geometry}
\usepackage{authblk}
\usepackage[onehalfspacing]{setspace}
\usepackage[T1]{fontenc}
\usepackage{hyperref}
\usepackage{url}

\usepackage{lipsum}
\usepackage{amsfonts}
\usepackage{graphicx}
\usepackage{epstopdf}
\usepackage{algorithmic}
\ifpdf
  \DeclareGraphicsExtensions{.eps,.pdf,.png,.jpg}
\else
  \DeclareGraphicsExtensions{.eps}
\fi

\usepackage[T1]{fontenc}
\usepackage[utf8]{inputenc}
\usepackage[english]{babel}
\usepackage{csquotes}

\usepackage{enumitem}
\usepackage{multirow}
\usepackage{comment}

\usepackage{booktabs}
\usepackage{caption} 
\usepackage{amsmath, bm, amssymb}
\usepackage{mathtools}
\usepackage{amsfonts}
\usepackage{amsthm}

\newtheorem{theorem}{Theorem}
\newtheorem{lemma}[theorem]{Lemma}
\newtheorem{proposition}[theorem]{Proposition}

\theoremstyle{definition}
\newtheorem{definition}[theorem]{Definition}
\newtheorem{remark}{Remark}

\newtheorem{assumption}{Assumption}

\newtheorem{example}{Example}

\usepackage{tikz}
\usepackage{graphicx}
\usepackage{float}
\usepackage{xcolor}

\usepackage{pgfplots}

\pgfdeclarelayer{ft}
\pgfdeclarelayer{bg}
\pgfsetlayers{bg,main,ft}
\pgfplotsset{compat=1.10}
\usepgfplotslibrary{fillbetween}
\usetikzlibrary{patterns}

\DeclareMathOperator{\Ao}{\mathbf{A}}
\DeclareMathOperator{\Fo}{\mathbf{F}}

\DeclareMathOperator{\Ro}{\mathbf{R}}

\DeclareMathOperator{\Bo}{\mathbf{B}}

\DeclareMathOperator{\To}{\mathbf{T}}

\DeclareMathOperator{\Mo}{\mathbf{M}}

\newcommand{\uo}{\mathbf u}
\newcommand{\vo}{\mathbf v}

\newcommand{\kao}{{\boldsymbol{\kappa}}}

\newcommand{\X}{\mathbb{X}}
\newcommand{\Y}{\mathbb{Y}}
\newcommand{\N}{\mathbb{N}}
\newcommand{\R}{\mathbb{R}}
\newcommand{\Z}{\mathbb{Z}}

\newcommand{\dom}{\operatorname{dom}}

\newcommand{\Rc}{\mathcal{R}}

\newcommand{\Fc}{\mathcal{F}}

\newcommand{\Qc}{\mathcal{Q}}
\newcommand{\ran}{\operatorname{ran}}

\newcommand{\la}{\lambda}
\newcommand{\La}{\Lambda}
\newcommand{\ka}{\kappa}

\newcommand{\al}{\alpha}
\newcommand{\ph}{\varphi}

\newcommand{\sign}{\operatorname{sign}}

\newcommand{\im}{\operatorname{im}}
\newcommand{\id}{\operatorname{Id}}
\newcommand{\argmin}{\operatorname{argmin}}

\newcommand{\diff}{\operatorname{d}}
\newcommand{\prox}{\operatorname{Prox}}

\newcommand\abs[1]{\left\vert#1\right\vert}
\newcommand\sabs[1]{{\lvert#1\rvert}}

\newcommand\norm[1]{{\left\Vert#1\right\Vert}}
\newcommand\snorm[1]{\Vert#1\Vert}

\newcommand\inner[2]{\left\langle#1,#2\right\rangle}
\newcommand\sinner[2]{\langle#1,#2\rangle}

\newcommand{\mpi}{+}
\newcommand{\qone}{s}

\title{Error Estimates for Data-driven Weakly Convex Frame-based Image Regularization}

\author{Andrea Ebner}

\affil{Department of Mathematics, University of Innsbruck\authorcr
Technikerstrasse 13, 6020 Innsbruck, Austria
 \authorcr E-mail:  \texttt{andrea.ebner@uibk.ac.at}
 }

 \author{Matthias Schwab}
 
\affil{Department of Radiology, Medical University of Innsbruck\authorcr
Anichstraße 35, 6020 Innsbruck, Austria.
 \authorcr E-mail:  \texttt{matthias.schwab@uibk.ac.at}
 }

\author{Markus Haltmeier}

\affil{Department of Mathematics, University of Innsbruck\authorcr
Technikerstrasse 13, 6020 Innsbruck, Austria
 \authorcr E-mail:  \texttt{markus.haltmeier@uibk.ac.at}
 }

\usepackage[backend=biber,url=false,doi=false, isbn=false, abbreviate=false,  style=numeric,natbib=true,sorting=nty]{biblatex}
\begin{document}
\emergencystretch 2em

\maketitle

\begin{abstract}
Inverse problems are fundamental in fields like medical imaging, geophysics, and computerized tomography, aiming to recover unknown quantities from observed data. However, these problems often lack stability due to noise and ill-conditioning, leading to inaccurate reconstructions. To mitigate these issues, regularization methods are employed, introducing constraints to stabilize the inversion process and achieve a meaningful solution.
Recent research has shown that the application of regularizing filters to diagonal frame decompositions (DFD) yields regularization methods. These filters dampen some frame coefficients to prevent noise amplification. This paper introduces a non-linear filtered DFD method combined with a learning strategy for determining optimal non-linear filters from training data pairs.
In our experiments, we applied this approach to the inversion of the Radon transform using 500 image-sinogram pairs from real CT scans. Although the learned filters were found to be strictly increasing, they did not satisfy the non-expansiveness condition required to link them with convex regularizers and prove stability and convergence in the sense of regularization methods in previous works. Inspired by this, the paper relaxes the non-expansiveness condition, resulting in weakly convex regularization. Despite this relaxation, we managed to derive stability, convergence, and convergence rates with respect to the absolute symmetric Bregman distance for the learned non-linear regularizing filters. Extensive numerical results demonstrate the effectiveness of the proposed method in achieving stable and accurate reconstructions. 

\end{abstract}

\section{Introduction}

Let $\Ao \colon \X \to \Y$ be a bounded linear operator mapping between two real Hilbert spaces $\X$ and $\Y$. Our focus is on the inverse problem of reconstructing $x^\mpi \in \X$ from noisy data
\begin{equation} \label{eq:ip}
y^\delta = \Ao x^\mpi + z \,,
\end{equation}
where $z$ denotes the data perturbation with $\snorm{z} \leq \delta$ for a given noise level $\delta > 0$. The inversion of the operator $\Ao$ is frequently ill-posed, as evidenced by the discontinuity of the Moore-Penrose inverse $\Ao^\mpi$. Consequently, exact solution methods for $\Ao x = y$ can significantly amplify small measurement errors in the data. To tackle this challenge, regularization methods have been devised with the goal of discovering approximate yet stable solution strategies \cite{Be17,En96,Ri03}.

\subsection{Filter-based Regularization}

Recently, there has been significant research on diagonal frame decompositions (DFD) combined with regularizing filters, presenting an efficient regularization approach for the inverse problem \eqref{eq:ip}. Suppose $\Ao$ has a diagonal frame decomposition (according to \cite{Eb23}) providing the representations:
\begin{align*}
\Ao &= \sum_{\lambda \in \Lambda} \kappa_\lambda \langle \cdot, u_\lambda \rangle \overline{v}_\lambda, \\
\Ao^\mpi &= \sum_{\lambda \in \Lambda} \kappa_\lambda^{-1} \langle \cdot, v_\lambda \rangle \overline{u}_\lambda,
\end{align*}
where $(u_\lambda)_{\lambda \in \Lambda}$ and $(v_\lambda)_{\lambda \in \Lambda}$ are frames of $\ker(\Ao)^\bot$ and $\overline{\ran(\Ao)}$, respectively, with corresponding dual frames $(\overline{u}_\lambda)_{\lambda \in \Lambda}$ and $(\overline{v}_\lambda)_{\lambda \in \Lambda}$.

Frame decompositions, a more general concept than singular value decomposition (SVD), were initially studied by Candés and Donoho \cite{Ca02, Do95} in the context of statistical estimation. Recent investigations, such as \cite{frikel2018efficient, Fr19, Eb23, goppel2023translation, quellmalz2023frame, Hu21_2}, explore the utility of frame decomposition for regularizing inverse problems.
It has been established in \cite{Eb23} that if $(\kappa_\lambda)_{\lambda \in \Lambda}$ converge to zero, then $\Ao^\mpi$ becomes unbounded, necessitating the application of regularization methods to address the solution of \eqref{eq:ip}.

Filtered DFD methods incorporate a regularizing filter $(\ph_\alpha)_{\alpha>0}$ to damp some frame coefficient, resulting in the reconstruction method:
\begin{equation} \label{eq:non-linear}
    \Fc_\alpha(y^\delta) \coloneqq \sum_{\lambda \in \Lambda} \kappa_\lambda^{-1} \ph_\alpha(\kappa_\lambda, \langle y^\delta, v_\lambda \rangle) \overline{u}_\lambda.
\end{equation}
The reconstruction mappings are designed to reduce noise amplification by diminishing coefficients with damping factors $\ph_\alpha(\kappa_\lambda, \langle y^\delta, v_\lambda \rangle)$ before inversion.
Key theoretical questions concern the stability of $\Fc_\alpha$ and its convergence to noiseless solutions as $\delta$ and $\alpha$ approach zero.

The analysis of filter methods is well established for linear filters in the SVD case \cite{En96, groetsch1984theory}. Choosing $\ph_\alpha(\kappa_\lambda, \langle y^\delta, v_\lambda \rangle) = (f_\alpha(\kappa_\lambda) \kappa_\lambda) \cdot \langle y^\delta, v_\lambda \rangle$ for a family $(f_\alpha)_{\alpha> 0}$ reduces the linear reconstruction operator $\Fc_{\alpha}(y^\delta) \coloneqq \sum_{\lambda \in \Lambda} f_\alpha(\kappa_\lambda) \langle y^\delta, v_\lambda \rangle \overline{u}_\lambda$. For the convergence analysis of the frame case, see \cite{Eb23}.
Linear regularizing filters depend only on quasi-singular values and are not data-dependent, which limits their performance. In practical situations, filters that include non-linear dependencies on $\langle y^\delta, v_\lambda \rangle$ typically offer better noise-filtering abilities than linear approaches; see \cite{An21, Ga98, Ma09}.

In the general case, $(\ph_\alpha)_{\alpha>0}$ is usually called a non-linear regularizing filter due to the non-linear dependence of the data coefficients in the damping process, resulting in a non-linear reconstruction method. For certain non-linear filters, it has been shown that they yield a convergent regularization method, such as the soft thresholding filter in \cite{Fr19}.
In \cite{ebner2023convergence}, filtered DFD methods were firstly studied for general non-linear filters, see Definition \ref{def:old_filter} . A convergence analysis was provided under the conditions that the functions $\ph_\alpha \colon \mathbb{R}^2 \to \mathbb{R}$ in the second component are non-expansive, increasing, converge monotonically to the identity as $\alpha \rightarrow 0$, and for small quasi-singular values, the behavior around zero can be controlled in a certain sense.

\subsection{Learned Regularization}

Recently, data-driven methods have become increasingly popular for solving ill-posed inverse problems. In this context, learned spectral regularization techniques for linear inverse problems have be studied \cite{bauermeister2020learning, kabri2022convergent}. However, to the best of our knowledge, no work has yet been done using learned non-linear regularizing filters $\ph_\alpha$ for solving inverse problems.  
In \cite{ebner2023convergence}, it was shown that under certain conditions on the filters, general non-linear frame-based diagonal filtering is a convergent regularization method. This provided us with the idea to check whether non-linear regularizing filters that are purely learned from data meet the theoretical requirements necessary for convergent regularization. 

In short, let $\Ao \colon \X \rightarrow \Y$ be a bounded linear forward operator with DFD $(\uo,\vo,\kao) = (u_\la,v_\la,\ka_\la)_{\la \in \La}$. Consider the learned reconstruction method 
\[ \Fc(y^\delta) = \sum_{\lambda \in \Lambda} \kappa_\lambda^{-1} \ph_{\hat{\theta}(\delta)}(\kappa_\lambda, \langle y^\delta, v_\lambda \rangle) \overline{u}_\lambda,  \]
where $\ph_\theta \colon \R^+ \times \R \rightarrow \R$ is a neural network with learnable parameters $\theta \in \Theta$ that is taken as non-linear filter. On a training dataset consisting of multiple data pairs of noisy measurements depending on $\delta$ and ground truth reconstructions, the neural network $(\ph_\theta)_  {\theta \in \Theta}$ is trained such that some reconstruction error is minimized, resulting in the learned filters $\ph_{\hat{\theta}(\delta)}$.  
\begin{figure*}[!t]
    \centering
    \includegraphics[width= \textwidth]{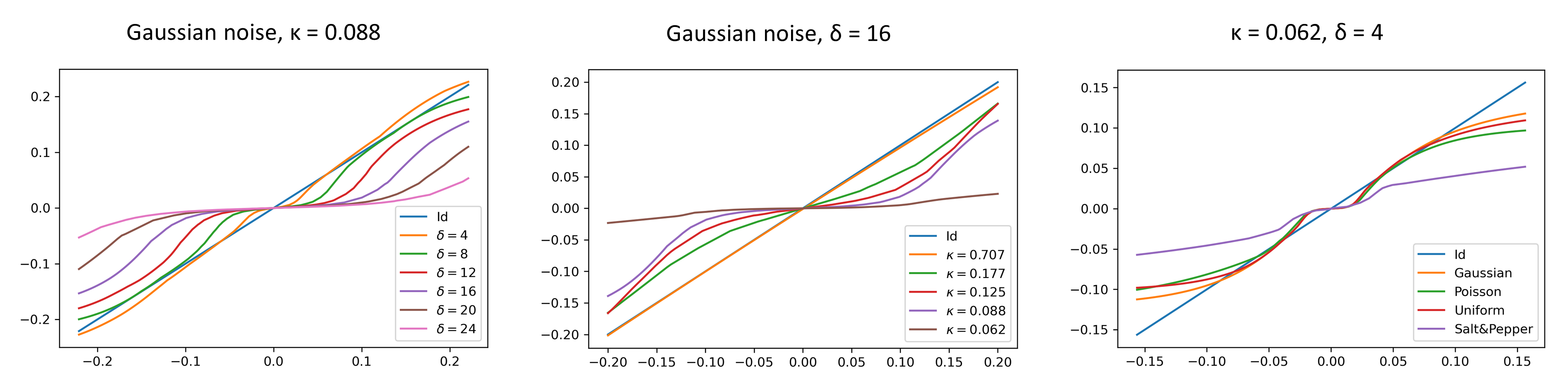}
    \caption{Learned filters for for different noise levels (left), different quasi-singular values (middle) and different noise types (right).}
    \label{fig:learned_filters}
\end{figure*}

In our experiments shown in section \ref{sec:learnt}, we learned non-linear regularizing filters for inverting the Radon transform. Using a DFD based on Haar wavelets the filters were learned from 400 image-sinogram pairs of real data CT scans \cite{soares2020sars}, where the sinograms were additionally corrupted with different types and levels of noise. Figure \ref{fig:learned_filters} shows examples of learned filters for different noise levels $\delta$, different quasi-singular values $\kappa$ and different noise types. One can clearly observe that although the filters are bijective and strictly increasing, and move towards the identity for smaller and smaller noise levels, they do not satisfy non-expansiveness, which is a precondition for convergence in \cite{ebner2023convergence}.
Inspired by this, in this paper, we will weaken the conditions of regularizing filters, providing stability and convergence for our reconstruction map while meeting the behavior of learned filter functions. 

\subsection{Outline}

The paper is structured as follows: The background section clarifies notional details and defines rigorously the non-linear filtered diagonal frame decomposition (DFD). For proof-related reasons as well as for a better geometric understanding, its connection to variational regularization is established, and further technical preparatory work is conducted. The elaboration of the experimental results is presented in section \ref{sec:learnt}, detailing implementation specifics. The theoretical main section \ref{sec_convReg} investigates the stability and convergence of the method and derives quantitative estimates of the convergence. The paper concludes with a brief summary discussing consistency of the experiments with the theory and outlines future directions.

\section{Auxiliary Results}

In this section, we introduce some basic notation and rigorously define our concept of the non-linear filtered diagonal frame decomposition (DFD), clarifying the conditions the filters must meet in this work. 
We explain the connection between filter-based and variational regularization and demonstrate that extending the conventional definition of proximity operators to certain non-convex functionals allows us to reduce the non-linear filtered DFD method, with relaxed conditions in contrast to \cite{ebner2023convergence}, to a still convex optimization problem. 
This not only has technical benefits but also provides a better geometric understanding of learned filter methods.

\subsection{Notation}

Let $\X$, $\Y$ be Hilbert spaces. For an operator $\Bo \colon \dom(\Bo) \subseteq \X \rightarrow \Y$ we denote $\dom(\Bo)$ as the domain and $\ran(\Bo)$ as the range of $\Bo$. In case of linear bounded operators, the domain is the entire space $\X$ and the Moore-Penrose inverse of $\Bo$ is denoted by $\Bo^+$ and is defined as $\Bo^+ \colon \dom(\Bo^+) \subseteq \Y \rightarrow \X$, where $\dom(\Bo^+)\coloneqq \ran(\Bo) \oplus \ran(\Bo)^\perp$.

Functionals on $\X$ will be written as $\Rc \colon \X \rightarrow \R \cup \{\infty\}$ and we  usually use $r$ or $\qone$ to denote a functional when $\X = \R$. We define the domain of  $\Rc$ by $\dom(\Rc) \coloneqq \{x \in \X \mid \Rc(x) < \infty \}$ and we call $\Rc$ proper if $\dom(\Rc) \neq \emptyset$. $\Rc$ is convex if $\Rc(t x + (1-t)y) \leq t \Rc(x) + (1-t) \Rc(y)$  for all $x , y \in \X$ and $t \in (0,1)$ and strictly convex if the strict inequality holds. We call $\Rc$ coercive if $\lim_{\norm{x}\rightarrow \infty} \Rc(x) = \infty$.

\subsection{Non-Linear Frame Filtering}

Consider a bounded linear operator $\Ao \colon \X \to \Y$ and an index set $\La$ that is at most countable.   We assume that $\Ao$ has a diagonal frame decomposition (DFD) $(\uo,\vo,\kao) = (u_\la,v_\la,\ka_\la)_{\la \in \La}$ as defined in \cite{Eb23}.

\begin{definition}[Diagonal Frame Decomposition, DFD] \label{def:dfd}
We call the triple $(\uo, \vo, \kao) = (u_\la, v_\la, \ka_\la)_{\la \in \La}$ a  diagonal frame decomposition (DFD) for $\Ao$ if the following holds:
\begin{enumerate}[itemindent =2em, leftmargin =1em,  label=(D\arabic*)]
\item\label{D1}  $(u_\la)_{\la \in \La}$ is a frame for $(\ker{\Ao})^{\perp} \subseteq \X$.
\item\label{D2}  $(v_\la)_{\la \in \La}$ is a frame for $\overline{\ran\Ao}\subseteq \Y$.
\item\label{D3}  $(\kappa_\la)_{\la \in \La}\in (0, \infty)^\La$ satisfies the quasi-singular relations $\forall \la \in \La \colon  \Ao^* v_\la = \ka_\la u_\la $.\end{enumerate}
We call $(\ka_\la)_{\la \in \La}$ the family of quasi-singular values and $(u_\la)_{\la \in \La}$, $(v_\la)_{\la \in \La}$  the   corresponding quasi-singular systems.
\end{definition}

DFDs are a broader concept than the singular value decomposition (SVD), essentially reducing to the SVD when $(u_\lambda)_{\lambda \in \Lambda}$ and $(v_\lambda)_{\lambda \in \Lambda}$ are orthonormal bases. The key advantage of DFDs over SVDs lies in their quasi-singular systems, which often provide better approximation properties. For instance, when $(u_\lambda)_{\lambda \in \Lambda}$ is chosen as the wavelet basis, as seen in applications like the Radon transform \cite{donoho1995nonlinear, Eb23, Hu21_2}. We leverage this approach for our numerical evaluation in Section \ref{sec:learnt}.

For a frame $\uo$ we define the synthesis and analysis operator of  $\bar \uo$ by 
\begin{align*}
\To_{\uo} \colon \ell^2(\La) \rightarrow \X \colon (c_\la)_{\la} \mapsto \sum_{\la \in \La} c_\la  u_\la \\
\To^*_{\uo} \colon \Y \rightarrow \ell^2(\La) \colon y \mapsto (\inner{y}{u_\la})_{\la \in \La}.
\end{align*}
Let   $(\uo, \vo, \kao) $  be a DFD for $\Ao$, and let $\bar \uo$ be a dual frame of $\uo$, defined  by  $x = \To_{\bar \uo} \To^*_{\uo} x$ for all $x \in \X$.  
Using the DFD, the Moore-Penrose inverse of $\Ao$ can be expressed as
\begin{equation} \label{eq:dfd-inv}
\forall y  \in \dom(\Ao^\mpi) \colon \quad
\Ao^\mpi (y) = \sum_{\la \in \La} \frac{1}{\ka_\la} \inner{y}{v_\la} \bar{u}_\la =  \To_{\bar \uo} \circ  \Mo^\mpi_\kao \circ \To_{\vo}^* (y)   \,,
\end{equation}
where $\Mo_{\kao}$ is the component-wise multiplication operator $\Mo_{\kao}((x_\la)_{\la \in \La}) = (\ka_\la x_\la)_{\la \in \La}$ and $\Mo^\mpi_\kao$ its Moore-Penrose inverse. As the frame operators $ \To_{\bar \uo}$ and $\To_{\vo}^* $ are continuous and invertible,  diagonalizing $\Ao$ with a DFD basically reduces the inverse problem \eqref{eq:ip} to an inverse problem with a diagonal forward operator from $\ell^2(\La)$ to $\ell^2 (\La)$.

Because inverting $\Ao$ is ill-posed, the values of $(\ka_\la)_{\la \in \La}$ accumulate at zero, implying that $(1/\ka_\la)_\la$ becomes unbounded.  Consequently, small errors in the data can be significantly amplified using \eqref{eq:dfd-inv}. To mitigate this error amplification, we employ regularizing filters with the goal of dampening noisy coefficients. In \cite{ebner2023convergence}, the following definition of a non-linear regularizing filter was given: 

\begin{definition}[Non-linear Regularizing Filter] \label{def:old_filter}
A family $(\ph_\al\colon \R_+ \times \R \to\R)_{\al > 0}$ is called a non-linear regularizing filter if for all $\al, \ka > 0$, the following holds
\begin{itemize}[itemindent =2em, leftmargin =1em]
\item  $\ph_\al(\ka, \cdot)$ is nonexpansive.
\item  $\ph_\al(\ka, \cdot)$ is monotonically increasing.
\item  $\ph_{\al}(\ka, 0)=0$.
\item  $ \forall c \in \R \colon \lim_{\al \to 0} \ph_{\al}(\ka,c) = c$.
\end{itemize}
\end{definition}

In this work, we use a slightly different definition of regularizing filters than Definition \ref{def:old_filter}. Instead of assuming the function $\ph_\alpha(\kappa, \cdot)$ to be nonexpansive, we are content with requiring bijectivity.  While this change seems minor, it has a major impact in practical usage. Nonexpansiveness is a highly restrictive condition, and especially learned filters usually do not meet this, while bijectivity and increasingness are satisfied without any additional constraints. The new definition of regularizing filters correspond to variational regularization with weakly convex penalty, illustrated in section \ref{sec:weakly_convex}. Therefore, we define a weakly convex regularizing filter as follows:

\begin{definition}[Weakly Convex Regularizing Filter] \label{def:filter}
We call a family $(\ph_\al\colon \R_+ \times \R \to\R)_{\al > 0}$ a weakly convex regularizing filter if for all $\al, \ka > 0$, the following holds
\begin{enumerate}[itemindent =2em, leftmargin =1em,  label=(F\arabic*)]
\item  \label{def:filter1} $\ph_\al(\ka, \cdot)$ is bijective.
\item  \label{def:filter2} $\ph_\al(\ka, \cdot)$ is strictly increasing.
\item  \label{def:filter3} $\ph_{\al}(\ka, 0)=0$.
\item  \label{def:filter4} $ \forall c \in \R \colon \lim_{\al \to 0} \ph_{\al}(\ka,c) = c$.
\end{enumerate}
\end{definition}

 In fact, the above requirements imply that $\ph_\al(\ka,\cdot)$ is a homeomorphism, being bijective, continuous, and  $\ph_\al(\ka,\cdot)^{-1}$ is continuous. 

Let $(\ph_\al)_{\al > 0}$ be a weakly convex regularizing filter, then we define by $\Fc_{\al}$ the non-linear filtered DFD 

\[ \Fc_\al (y^\delta) 
\coloneqq \sum_{\la \in \La} \frac{1}{\ka_\la}\ph_{\al}(\ka_\la,\sinner{y^\delta}{v_\la}) \bar u_\la = 
 \To_{\bar \uo} \circ  \Mo^+_\kao \circ \,\Phi_{\al, \kao} 
\circ \To_{\vo}^* y^\delta,  \]
where $\Phi_{\al,\kao} \colon \dom(\Phi_{\al,\kao}) \subseteq \ell^2(\La) \rightarrow \ell^2(\La) \colon (c_\la)_{\la \in  \La} \mapsto (\ph_{\al}(\ka_\la,c_\la))_{\la \in \La}$ is a non-linear diagonal operator.

The aim of this paper is to acquire the filter functions $\ph_\al$ from data, demonstrating the stability and convergence of the method $\Fc_\al$ under conditions that align with the learned filter behavior. This demonstrates, in particular, that data-driven non-linear spectral regularization is a convergent regularization method.

Given that the frame operators $\To_{\bar \uo}$ and $\To_\vo^*$ are linear and bounded, the examination of stability and convergence for the entire method is simplified to the analysis of the diagonal operator $\Mo^+_\kao \circ \,\Phi_{\al, \kao}$.

\subsection{Weakly Convex Regularization}\label{sec:weakly_convex}

An important ingredient of our analysis is the link of filter-based regularization methods to variational regularization. Recall that the objective of standard variational regularization is to minimize  $\mathcal{T}_{\al,y}(x) = \norm{\Ao x - y }^2/2 + \al \Rc(x)$, called the generalized Tikhonov functional, where the regularizing functional $\Rc$ is typically proper, convex, and lower semi-continuous.

In \cite{ebner2023convergence}, it has been demonstrated that if the filter functions $\ph_{\al}(\ka,\cdot)$ satisfy Definition \ref{def:old_filter} and $(\ph_\al(\ka,\cdot)^{-1}-\id)/\al$ is independent of $\al$, then the non-linear filtered DFD essentially simplifies to a specific variational regularization. Specifically, the evaluation of the diagonal part,  $\Mo^+_\kao \circ \,\Phi_{\al, \kao}$, is equivalent to minimizing $\norm{\Mo_{\kao} x - y }^2/2 + \al \Rc(x)$ with a proper, convex, and lower semi-continuous but potentially operator-dependent regularizer. 

Admitting the latter, supposing only that $\ph_{\al}(\ka,\cdot)$ is a non-linear regularizing filter by Definition \ref{def:old_filter}, also leads to a variational regularization with the regularization term  $\Rc_\al$ instead of a so-called stationary regularizer $\al \Rc$. Nonetheless,  $\Rc_\al$ remains proper, convex, and lower semi-continuous. In fact, their increasingness and nonexpansiveness are necessary conditions for the filter functions to be transformed into a optimization problem with convex penalty. These assumptions are not satisfied when learning filters from data. In Figure \ref{fig:learned_filters}, it is evident that the learned filters at least increase, but the slope exceeds 1, and at times, the filter surpasses the identity function. Still, we can turn the filtered DFD into a variational problem, but these characteristics mean that the regularization term might be non-convex or negative.

Reducing the filtered DFD method to variational regularization lies in the observation that increasing and nonexpansive filter functions serve as proximity operators for convex functions on $\R$. A key observation in this work is that it is possible to relax the nonexpansiveness assumption and use proximity operators with weakly convex penalties instead. 

\begin{definition}[Weakly Strictly Convex, Weakly Coercive]
A function $\qone \colon \R \rightarrow \R$ is called weakly (or 1-weakly) strictly convex if  $\qone + 1/2 \abs{ \cdot}^2$ is strictly convex, and $\qone$ is weakly coercive if $\qone + 1/2 \abs{ \cdot}^2$ is coercive.
\end{definition}

\begin{definition}[Generalized Proximity Operator] \label{Def:GProx}
Let $\qone\colon \R \rightarrow \R$ be weakly strictly convex and weakly coercive. We define the single-valued proximity operator of $\qone$ by
\[ \prox_\qone(x) \coloneqq \underset{\tilde{x} \in \R}{\argmin} \{\frac{1}{2} \abs{x- \tilde{x}}^2 +\qone(\tilde{x}) \}.\]
\end{definition}

The objective function of the generalized proximity operator is strictly convex and coercive, guaranteeing the existence of a unique minimizer. 
For more details on generalized proximity operators we refer to \cite{gribonval2020prox}, where the authors investigate the use of proximity operators with non-convex penalty, and \cite{Goujon2024weaklyconvex}, where learned weakly convex regularizers were used for image reconstruction.  It is worth noting that in some literature, the term "semiconvex" is used interchangeably with "weakly convex".

The following proposition demonstrate that a filter function is a proximity operator of a weakly strictly convex and weakly coercive function $\qone$, and additionally, $\qone$ is continuously differentiable.

\begin{proposition}\label{prop:prox}
The following assertions hold: 
\begin{enumerate}[itemindent =2em, leftmargin =1em]
\item Let $\ph \colon \R \rightarrow \R$ be bijective and strictly increasing function with $\ph(0)=0$. Then, $\ph = (\id +\qone')^{-1}(x)$, where $\qone \colon \R \rightarrow \R$ is continuously differentiable, $\qone(0)=0$, weakly  strictly convex and weakly coercive.
Especially, we can calculate $\qone$ by $\qone(x) = \int_0^x \ph^{-1}(y) \diff y - \frac{x^2}{2}$ for $x \in \R$. (Note that for negative $x$, we define $\int_{0}^x \ph 
 \coloneqq -\int_{x}^0 \ph$.)
\item Let $\qone \colon \R \rightarrow \R$ be continuously differentiable, $\qone(0)=0$, weakly strictly convex and weakly coercive. Then, $(\id + \qone')^{-1}$ is well defined and we have 
\[\prox_\qone(x)=(\id +\qone')^{-1}(x).\]
\end{enumerate}
\end{proposition}

\begin{proof}
\begin{enumerate}[itemindent =2em, leftmargin =1em]
\item Given that $\ph$ is bijective and strictly increasing, it is continuous. Hence, $\ph^{-1}$ is bijective, strictly increasing and continuous as well.  Define $\qone(x) \coloneqq \int_0^x \ph^{-1}(y) \diff y - \frac{x^2}{2}$ for all $x \in \R$. 
By the fundamental theorem of calculus, $\qone$ is continuously differentiable on $\R$ with $\qone'(x) = \ph^{-1}(x)-x$ for all $x \in \R$. It follows that $\ph(x) = (\id + \qone')^{-1}(x)$ for all $x \in \R$, and by definition, $s(0)=0$. Now, observe that 
$\ph^{-1}=\qone' + \id$ which is strictly increasing and bijective. This shows that $s+\abs{\cdot}^2/2$ is strictly convex and coercive.

\item Given that $s$ is weakly strictly convex and weakly coercive, $\id + s'$ is bijective. 
Further, we have 
\begin{align*}
\prox_s (x) = \underset{\tilde{x} \in \R}{\argmin} \{ \frac{1}{2} \abs{x- \tilde{x}}^2 +\qone(\tilde{x}) \} &\subseteq \{ \tilde{x} \in \R \mid (\tilde{x}-x) + \qone'(\tilde{x}) =0 \} \\
&= \{ \tilde{x} \in \R \mid \tilde{x} + \qone'(\tilde{x}) = x \} = (\id + \qone')^{-1}(x).
\end{align*}
\end{enumerate}
\end{proof}

Now we can show that the non-linear filtered DFD reduces to variational regularization with a general type of regularizing term, a so-called $\kao$-regularizer. 

\begin{definition}[$\kao$-Regularizer]
Let $(\ph_\al)_{\al>0}$ satisfy \ref{def:filter1}-\ref{def:filter3}. For all $\al > 0$ and $\la \in \La$ we define $\qone_{\al,\la}(x) \coloneqq \int_0^x \ph_\al(\ka_\la,\cdot)^{-1}(y) \diff y - x^2/2$. By Proposition \ref{prop:prox} the proximity operator of $\qone_{\al,\la}$ is well defined and we have $\ph_{\al}(\ka_\la,\cdot) = \prox_{\qone_{\al,\la}}$. We call $(\Rc_\al)_{\al >0}$ defined by $\Rc_\al(x) \coloneqq \sum_{\la \in \La} \qone_{\al,\la}(\ka_\la x_\la)$ the $\kao$-regularizer defined by the filter $(\ph_\al)_{\al>0}$ and the weight vector $\kao$. Further, we call a $\kao$-regularizer stationary if $\Rc_\al= \al \Rc_1$ and non-stationary otherwise.
\end{definition}

\begin{lemma}\label{lem_main}
Let $(\ph_\al)_{\al>0}$ satisfy \ref{def:filter1}-\ref{def:filter3}. 
Then, for all $z \in \dom(\Mo_{\kao}^+ \circ \Phi_{\al,\kao})$ we have
\begin{align*}
\Mo^+_\kao \circ \,\Phi_{\al, \kao} (z) & = \underset{x \in \ell^2}{\argmin} \{\frac{1}{2} \norm{\Mo_\kao x -z }^2 + \Rc_\al(x) \}.
\end{align*}
\end{lemma}

\begin{proof}
For all $\la \in \La$, we have 
\begin{align*}
\underset{x \in \R}{\argmin} \{ \frac{1}{2} \abs{\ka_\la x -z_\la}^2 + \qone_{\al,\la}(\ka_\la x)\} &= \frac{1}{\ka_\la}\underset{x \in \R}{\argmin} \{ \frac{1}{2} \abs{x -z_\la}^2 + \qone_{\al,\la}(x)\} \\ 
&= \frac{1}{\ka_\la} \prox_{\qone_{\al,\la}}(z_\la) = \frac{1}{\ka_\la} \ph_\al(\ka_\la,z_\la),
\end{align*}
which is a unique minimizer.
Since $\Mo_\kao^+ \circ \Phi_{\al, \kao} (z) \in \ell^2(\La)$, we have 
\begin{align*}
\Mo_\kao^+ \circ \Phi_{\al, \kao} (z)
&=\left( \frac{1}{\ka_\la} \ph_\al(\ka_\la,z_\la)\right)_{\la \in \La}\\
&= \left(\underset{x_\la \in \R}{\argmin}\{ \frac{1}{2} \abs{\ka_\la x_\la -z_\la }^2 + \qone_{\al,\la}(\ka_\la x_\la) \} \right)_{\la \in \La} \\
&= \underset{(x_\la)_{\la \in \La} \in \ell^2(\La)}{\argmin} \{ \sum_{\la \in  \La} \frac{1}{2} \abs{\ka_\la x_\la -z_\la }^2 + \qone_{\al,\la}(\ka_\la x_\la) \} \\
&=\underset{x \in \ell^2}{\argmin} \{\frac{1}{2} \norm{\Mo_\kao x -z }^2 +\Rc_\al(x) \}. 
\end{align*}
\end{proof}

Note that, since the $\Rc_\al$ involves the quasi-singular values $\kao$, the regularizing term, unlike classical variational regularization, depends on the operator $\Ao$. However, this connection gives us a better understanding of how damping of frame coefficients affects the regularization process. Furthermore, $\kao$-regularizers are separable functionals composed of an infinite sum of weakly convex differentiable functions. As a result, we can demonstrate that they are Gateaux differentiable, with the proof provided in the Appendix.

\begin{proposition}[Differentiability of $\Rc_\al$] \label{prop:diff_R}
Let $\al>0$ and define 
$\dom(\Rc_\al') \coloneqq \{x \in \ell^2(\La) \mid (\ka_\la \qone_{\al,\la}'(\ka_\la x_\la))_\la \in \ell^2(\La) \}.$
Then, $\Rc_\al$ is Gateaux differentiable on $\dom(\Rc_\al)^\circ \, \cap \, \dom(\Rc_\al')$ with gradient ${\nabla \Rc_\al(x) = (\ka_\la \qone_{\al,\la}'(\ka_\la x_\la))_\la}$.
\end{proposition}

Example \ref{ex:like_learned} considers filter functions that are artificially constructed such that they align with our experimental evaluation and investigates how they affect the one-dimensional regularizing functions $\qone_{\al,\la}$.

\begin{figure}[htb!]
\centering
\begin{tikzpicture}[x=1.5cm, y=1.5cm]
\draw[->] (-0.5, 0) -- (3, 0) node[right] {$x$};
\draw[->] (0, -0.5) -- (0, 3) node[left] {$\ph_{\al}(\ka_\la,\cdot)$};
\draw[dashed] (-0.5, -0.5) -- (3, 3);
\draw[domain=-0.5:1/2,variable=\x, line width=0.4mm, pink] plot ({\x},{\x*\x*\x*3});
\draw[domain=1/2:3,variable=\x, line width=0.4mm, pink] plot ({\x},{1/8*9^(2/3)*3*((\x-7/18)^(2/3)-(1/2-7/18)^(2/3)) +(1/8)*3});
\draw[domain=-0.5:1,variable=\x, line width=0.4mm,  brown] plot ({\x},{\x*\x*\x*3/4});
\draw[domain=1:3,variable=\x, line width=0.4mm,  brown] plot ({\x},{sqrt(6)*3/4*((\x-5/6)^(1/2)-(1-5/6)^(1/2))+3/4});
\draw[domain=-0.5:2,variable=\x, line width=0.4mm,  red] plot ({\x},{\x*\x*\x*1/8});
\draw[domain=2:3,variable=\x, line width=0.4mm,  red] plot ({\x},{(2/9)^(-1/3)*((\x-16/9)^(1/3)-(2-16/9)^(1/3)) +1});
\draw[domain=-0.5:3,variable=\x, line width=0.4mm,  orange] plot ({\x},{\x*\x*\x*1/32});
\end{tikzpicture}
\hfill 
\begin{tikzpicture}[x=1.5cm, y=1.5cm]
\draw[->] (-0.5, 0) -- (3, 0) node[right] {$x$};
\draw[->] (0, -0.5) -- (0, 3) node[left] {$\qone_{\al,\la}$};
\draw[dashed] (-0.5, -0.5) -- (3, 3);
\draw[domain=-0.5:(3/8),variable=\x, line width=0.4mm,  pink] plot ({\x},{(-\x*\x*1/2+(3/4)*(1/3)^(1/3)*(\x*\x)^(2/3))});
\draw[domain=(3/8):3,variable=\x, line width=0.4mm,  pink] plot ({\x},{2/5*3/8*(9^2)^(1/3)*(((\x-3/8)*8/3*(1/9)^(2/3)+((1/9)^2)^(1/3))^(5/2)-((1/9))^(5/3))+7/18*(\x-3/8)-1/2*(\x^2-(3/8)^2)+(-(3/8)^2*1/2+(3/4)*(1/3)^(1/3)*((3/8)^2)^(2/3))});
\draw[domain=-0.5:(3/4),variable=\x, line width=0.4mm,  brown] plot ({\x},{(-\x*\x*1/2+(3/4)*(4/3)^(1/3)*(\x*\x)^(2/3))});
\draw[domain=(3/4):3,variable=\x, line width=0.4mm,  brown] plot ({\x},{sqrt(6)/4*(((\x-3/4)*4/(3*sqrt(6)) +1/sqrt(6))^3-(1/sqrt(6))^3)+5/6*(\x-3/4)-1/2*(\x^2-(3/4)^2)+(-(3/4)^2*1/2+(3/4)*(4/3)^(1/3)*((3/4)^2)^(2/3))});
\draw[domain=-0.5:1,variable=\x, line width=0.4mm,  red] plot ({\x},{(-\x*\x*1/2+(3/4)*8^(1/3)*(\x*\x)^(2/3))});
\draw[domain=1:2.4,variable=\x, line width=0.4mm,  red] plot ({\x},{1/4*((2/9)^(-1/3))*(((\x-1)*((2/9)^(1/3))+(2/9)^(1/3))^4-(((2/9)^2)^(2/3)))+16/9*(\x-1)-1/2*(\x^2-1)+(-1/2+(3/4)*8^(1/3))});
\draw[domain=-0.5:27/32,variable=\x, line width=0.4mm,  orange] plot ({\x},{(-\x*\x*1/2+(3/4)*32^(1/3)*(\x*\x)^(2/3))});
\draw[domain=27/32:1.4,variable=\x, line width=0.4mm,  orange] plot ({\x},{(27/32)*(sqrt(2)/5)*(((\x-27/32)*(32/(27*sqrt(2)))+(1/sqrt(2)))^5-(1/sqrt(2))^5)+11/4*(\x-(27/32))-1/2*(\x^2-(27/32)^2)+ (-(27/32)^2*1/2+(3/4)*32^(1/3)*((27/32)^2)^(2/3))});
\end{tikzpicture}
\caption{On the left filter functions of example \ref{ex:like_learned} for different $\alpha$ and fixed $\kappa$ are plotted, along with their corresponding regularizing functions $\qone_{\al,\la}$ on the right. These are inspired by the learned filter functions in Figure \ref{fig:learned_filters}. } \label{fig:like_learned}
\end{figure}
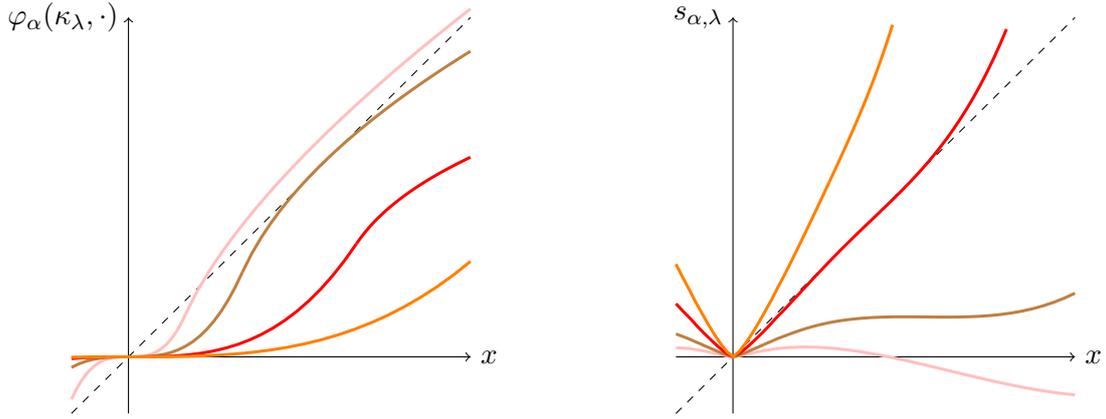 

\begin{example}\label{ex:like_learned}
Let $\ka_\la$ be fixed. Consider the filter function
\begin{align*}
    \ph_\al(\ka_\la, x) =  \begin{cases}
        \frac{1}{3\al^2} x^3 & \abs{x} \leq \al  \\
        \frac{\al}{3} \left(\frac{3(1 +\al)}{\al} x- (2+ 3\al)\right)^{\frac{1}{\al + 1}} & \text{ otherwise. }
    \end{cases}
\end{align*}
This filter function fulfills the conditions \ref{def:filter1}-\ref{def:filter4}. It is inspired by the learned filter functions shown in Figure \ref{fig:learned_filters} and is plotted in Figure \ref{fig:like_learned} on the left for different parameters $\alpha$. On the right, the corresponding regularizing function $\qone_{\alpha,\lambda}$ is plotted. We observe that the regularizing function does not have to be convex and it is possible for it to become negative.
\end{example}

\section{Learned Filters} \label{sec:learnt}

\subsection{DFD for the Radon Transform}
\label{sec:DFD}
We examine the behavior of learned regularizing filters using the example of the two dimensional Radon transform $\Ro: L^2(B_1(0)) \rightarrow L^2(S^1 \times \R)$, which is defined by 
\[ \Ro f(\omega,s) \coloneqq \int_\R f(s \omega +t \omega^\perp) dt, \quad (\omega,s) \in S^1 \times \R, \]
where $B_1(0) = \{x \in \R^2 : \norm{x} \leq 1 \}$ and $S^1 = \{x \in \R^2 : \norm{x} =1 \}$. Consider a two dimensional wavelet orthonormal bases  $(w_\la)_{\la \in \La}$ of $L^2(\R)$ with compact support. The index $\la$ consists of three components $(j,k,\beta)$, with scale index $j \in \Z$, shift index $k \in \Z^2$ and $\beta \in \{1,2,3\}$ indicating the three mother wavelets. By defining 
\[ u_\la(x)\coloneqq \begin{cases} 
    w_\la(x) \quad & x \in B_1(0)\\
    0 \quad & \text{otherwise,} 
\end{cases}\]
one gets a self dual tight frame for $L^2(B_1(0))$. Further, the standard filtered backprojection formula (FBP) for inverting the Radon transform is given by 
\begin{align*}
     \operatorname{FBP}(g) \coloneqq (4 \pi)^{-1} (\Ro^* \circ \mathcal{I}_1)(g), 
\end{align*}
with Riesz potential $\mathcal{I}_1(g) \coloneqq \Fo^{-1}( |\cdot | \Fo g)$, where $\Fo$ denotes the Fourier transform. It can be shown \cite{donoho1995nonlinear} that when defining 
\begin{align*}
v_\la &\coloneqq 2^{j/2} (4 \pi)^{-1} \mathcal{I}_1 \Ro u_\la \\
\kappa_\la &\coloneqq 2^{j/2}
\end{align*}
$(u_\la, v_\la, \kappa_\la)_{\la \in \La}$ forms a DFD for the two dimensional Radon transform.  Making use of the fact that 
\[ \sinner{g}{v_\la}= \ka_\la \sinner{\operatorname{FBP(g)}}{u_\la}, \]
for $g \in L^2(S^1 \times \R)$ we get the reconstruction method
\begin{align}
\label{eq:inv}
    \Fc(g) = \sum_{\la \in \La} \frac{1}{\kappa_\la} \ph(\kappa_\la, \ka_\la \sinner{\operatorname{FBP}(g)}{u_\la}) u_\la,
\end{align}
where $\ph: \R^+ \times \R \rightarrow \R$ is a weakly convex regularizing filter. Our goal is now to learn these filters in a data-driven way and thus to represent them using neural networks.

\subsection{Training Details}

In our experiments, we investigated data-driven frame-based image reconstruction of the two dimensional Radon transform, meaning we learned the non linear regularizing filters in equation \eqref{eq:inv} directly from data. For this purpose, we used a publicly available dataset \cite{soares2020sars} of real patients CT scans which was originally built for SARS-CoV-2 identification. From this dataset, we extracted a training data set of $N= 400$ CT images of healthy patients. We chose equidistant discretizations $\omega_1, \dots \omega_{512}, s_1, \dots s_{363}$ and calculated the sinograms $y_i= (\Ro x_i (\omega_j, s_k))_{j,k=1}^{512,363} \in\R^{512 \times 363}$ of the training images $x_i \in \R^{256 \times 256}$. Further, we corrupted them with noise of different noise levels $\delta$, obtaining noisy sinograms $y_i^\delta$.

To get a DFD of the two dimensional Radon transform, we chose the Haar wavelets and defined $(u_\la, v_\la, \kappa_\la)_{\la \in \La}$ as explained in section \ref{sec:DFD}. Now consider a neural network with $\varphi_\theta: \R^+ \times \R \rightarrow \R$ with learnable parameters $\theta \in \R^m$. Our goal is to find a reconstruction method such that the mean squared reconstruction error is minimized on the training dataset. This means for a given noise level $\delta$, we have to solve the minimization problem:
\begin{align*}
\underset{\theta \in \R^m}\argmin \frac{1}{N} \sum_{i=1}^N \norm{x_i - \Fc_\theta(y_i^\delta)}_2^2,   
\end{align*}
where
\begin{align*}
    \Fc_\theta(y^\delta) =  \sum_{\la \in \La} \frac{1}{\kappa_\la} \ph_\theta(\kappa_\la, \ka_\la \sinner{\operatorname{FBP}(y^\delta)}{u_\la}) u_\la.
\end{align*}
Note that because of numerical considerations, we do not learn the weakly convex regularizing filters directly as they appear in formula \eqref{eq:inv}. More precisely we set
 \[
\ph_\theta( \ka_\la, x) \coloneqq \ka_\la \psi_\theta(\ka_\la, x/\ka_\la)
 \]
and for each noise level during training minimize
\begin{align}\label{eq:loss}
    E_\delta(\theta) = \frac{1}{N} \sum_{i=1}^N \norm{x_i -\sum_{\la \in \La} \psi_\theta(\ka_\la, \sinner{\operatorname{FBP}(y_i^\delta)}{u_\la}) u_\la}^2,
\end{align}
where $\psi_\theta: \R^+ \times \R \rightarrow \R$ is a neural network consisting of five fully connected layers. We trained the networks for $100$ epochs and validated the performance of the models after each epoch on a validation set of $100$ images. The final network was chosen as the network with the best validation loss over the $100$ epochs. The deep learning framework was implemented in Python using the PyTorch library. For optimization, we used the AdamW optimizer with an exponentially decaying learning rate. After approximate minimizers $\hat{\theta}(\delta)$ of \eqref{eq:loss} were found during training phase, the final regularization method is given by
\[ \Fc_{\hat{\theta}(\delta)}(y^\delta) = \sum_{\la \in \La} \frac{1}{\kappa_\la} \ph_{\hat{\theta}(\delta)}(\kappa_\la, \ka_\la \sinner{\operatorname{FBP}(y^\delta)}{u_\la}) u_\la.\]
The full Python code for learning the weakly convex regularizing filters will be publicly available at \url{https://github.com/matthi99/learned-filters.git}.

 \begin{figure}
    \centering
    \includegraphics[width = .8\textwidth]{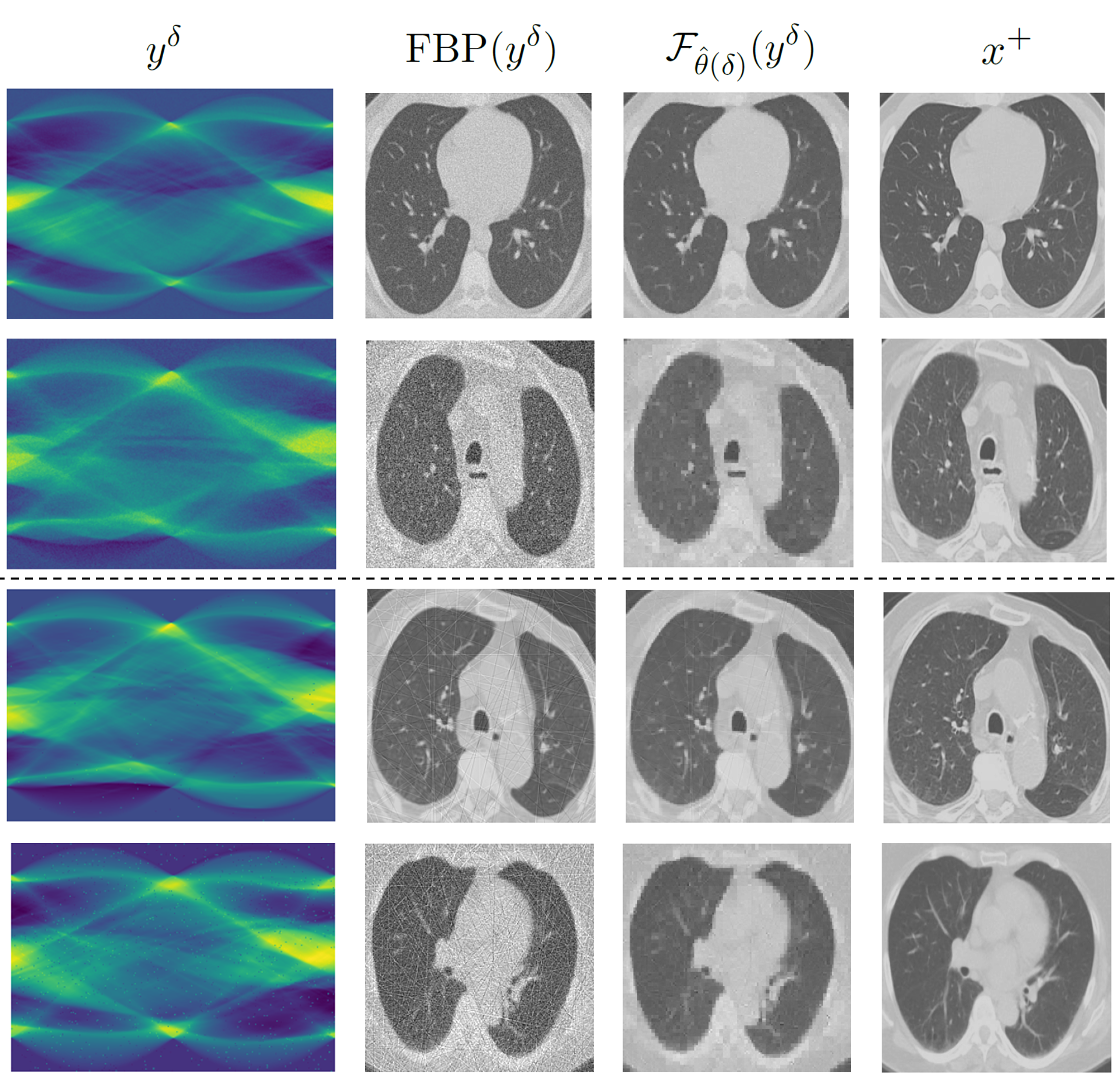}
\caption{Comparison of ground truth images $x^+$, filtered back projections $\operatorname{FBP}(y^\delta)$, and learned reconstructions $\Fc_{\hat{\theta}(\delta)}(y^\delta)$. Sinograms $y^\delta$ were corrupted with Gaussian noise (top two rows) and Salt$\&$Pepper noise (bottom two rows) with noise levels $\delta = 4$ (rows 1 and 3) and $\delta = 12$ (rows 2 and 4).  }
    \label{fig:example}
\end{figure}

\begin{table}[!b]
    \centering
    \begin{tabular}{l c c c c c c c c}
    Noise       & \multicolumn{2}{c}{Gaussian}         & \multicolumn{2}{c}{Poisson}        & \multicolumn{2}{c}{Uniform}       & \multicolumn{2}{c}{Salt\&Pepper}\\
                & FBP  & $\Fc_{\hat{\theta}(\delta)}$  & FBP &$\Fc_{\hat{\theta}(\delta)}$  & FBP &$\Fc_{\hat{\theta}(\delta)}$ & FBP &$\Fc_{\hat{\theta}(\delta)}$\\
    \midrule
    $\delta =4$ & $0.008$ &$0.004$                   & $0.009$ & $0.004$                & $0.008$ & $0.004$               & $0.007$ & $0.005$ \\
    $\delta =8$ & $0.039$ &$0.008$                   & $0.044$ & $0.008$                & $0.039$ & $0.007$               & $0.037$ & $0.008$ \\
    $\delta =12$& $0.095$ &$0.011$                   & $0.107$ & $0.012$                & $0.095$ & $0.011$               & $0.092$ & $0.012$ \\
    $\delta =16$& $0.175$ &$0.015$                   & $0.198$ & $0.016$                & $0.175$ & $0.014$               & $0.169$ & $0.015$ \\
    $\delta =20$& $0.279$ &$0.017$                   & $0.316$ & $0.018$                & $0.280$ & $0.017$               & $0.270$ & $0.018$ \\
    $\delta =24$& $0.408$ &$0.020$                   & $0.461$ & $0.021$                & $0.408$ & $0.020$               & $0.394$ & $0.020$ \\
    $\delta =28$& $0.560$ &$0.022$                   & $0.634$ & $0.023$                & $0.561$ & $0.022$               & $0.541$ & $0.023$ \\
    \end{tabular}
    \caption{Mean MSE over a test dataset of 50 images for different noise types and noise levels.}
    \label{tab:MSE}
\end{table}

\subsection{Results}
In the numerical experiments, we tested different types of noise like Gaussian noise, Poisson noise, uniformly distributed noise, and salt and pepper noise. We also investigated different noise levels $\delta \coloneqq 1/\sqrt{n} \norm{y-y^\delta}_2$, where $n$ is the amount of sinogram pixels, and $\delta \in \{4,8,12,16,20,24,28\}$. Figure \ref{fig:example} shows four examples of learned filter-based reconstructions compared with normal FBP reconstructions and ground truth images $x^+$. The mean squared errors (MSE) achieved on a test set of 50 images for different types of noise and noise levels are displayed in Table \ref{tab:MSE}. It can be seen that for all noise types and levels, the MSE is significantly lower for the learned filter-based reconstructions compared to the normal FBP reconstructions.  

\section{Convergence Analysis} \label{sec_convReg}

In this section, we analyze the stability and convergence of the learned non-linear filtered DFD and provide quantitative estimates based on the absolute symmetric Bregman distance. Let $(\ph_\al)_{\al>0}$ be a weakly convex regularizing filter (see Definition \ref{def:filter}) and the non-linear filtered DFD is given by: 
\begin{equation}
    \Fc_\al (y) 
\coloneqq \sum_{\la \in \La} \frac{1}{\ka_\la}\ph_{\al}(\ka_\la,\sinner{y}{v_\la}) \bar u_\la = 
 \To_{\bar \uo} \circ  \Mo^+_\kao \circ \,\Phi_{\al, \kao} 
\circ \To_{\vo}^* y.
\end{equation}
Our analysis focuses on the family of diagonal operators $(\Mo_\kao^+ \circ \Phi_{\al,\kao})_{\al > 0}$, leveraging the continuity of $\To_{\bar \uo}$ and $\To_{\vo}^\ast$. To establish stability and convergence, we introduce additional assumptions on the weakly convex regularizing filter $(\ph_\al)_{\al>0}$, ensuring its smallness in a neighborhood of zero relative to $\ka$ as $\al$ tends to zero. 
Additionally, for establishing a convergence rate, we assume that the $\kao$-regularizer $\Rc_\al$ defined by $(\ph_\al)_{\al>0}$ lies within a specific neighborhood of a stationary $\kao$-regularizer $\al \Qc$.

\subsection{Convergent Regularization Method}

We show that the non-linear filtered DFD is a convergent regularization method in a weak sense. This means that for a fixed $\al>0$, the reconstruction process $\Fc_\al$ is weak continuous, and for $y^k \rightarrow y \in \ran(\Ao)$ and a certain sequence $\al_k \rightarrow 0$ we have $\Fc_{\al_k}(y^k) \rightharpoonup \Ao^+ y$.

\begin{assumption}\label{Ass}
Let $(\ph_{\al})_{\al > 0}$ be weakly convex regularizing filter which satisfies the following conditions:
\begin{enumerate} [label=(A\arabic*), leftmargin =2.5em]
\item \label{ass:A1} $ \exists c, d> 0$ $\forall \ka > 0$  $ \forall \al > 0$ $ \forall x \in \R \colon
\sabs{x} \leq c \al / \ka \Rightarrow \abs{\ph_{\al}(\ka,x)} \leq d \ka/\sqrt{\al} \sabs{x}$.
\item \label{ass:A3} $\exists K>0$ $\forall\ka > 0$ $ \forall \al > 0$ $\forall x \in \R \colon  \abs{\ph_{\al}(\ka,x)} \leq K \sabs{x}$.
\item \label{ass:A2} $\forall\ka > 0$ $ \exists \tilde \al > 0 $ $ \exists g_\ka \colon [0,\infty) \rightarrow [0,\infty)$ bicontinuous such that
\[\forall \al \in (0,\tilde \al) \, \forall x \in \R \colon \abs{\ph_{\al}(\ka,x)} \geq g_\ka(\abs{x}).\]
\end{enumerate}
\end{assumption}

\begin{remark}[Impact of Assumption \ref{Ass} to $\qone_{\al,\la}$] \label{rem:Ass_A}
Suppose assumption \ref{ass:A1} holds. Then, by the definition of the generalized proximity operator, applying \ref{ass:A1} to $\abs{\ph(\ka_\la, \cdot)^{-1}(x)}$ the following implication holds:
\[ \abs{x+\qone'_{\al,\la}(x)} \leq \frac{c \al}{\ka_\la} \Rightarrow \abs{x+\qone'_{\al,\la}(x)} \geq \frac{\sqrt{\al}}{d \ka_\la} \abs{x}.\]
Integration on both sides then results in 
\[ \frac{x^2}{2}+\qone_{\al,\la}(x) \leq \frac{c \al}{\ka_\la} \abs{x} \Rightarrow \frac{x^2}{2} +\qone_{\al,\la}(x) \geq \frac{\sqrt{\al}}{d \ka_\la} \frac{x^2}{2}.\]
At $x= \pm 2 c d \sqrt{\al}$, the two curves $c\al \abs{x}/\ka_\la$  and $\sqrt{\al} x^2 /(2 d \ka_\la)$ intersect, and therefore we have that 
\[
\begin{cases}
    \frac{x^2}{2}+\qone_{\al,\la}(x) \geq \frac{\sqrt{\al} x^2}{2 d \ka_\la}& \abs{x} < 2 c d \sqrt{\al} \\
    \frac{x^2}{2}+\qone_{\al,\la}(x) \geq \frac{c \al}{\ka_\la} \abs{x}& \abs{x} \geq 2 c d \sqrt{\al}. 
\end{cases}
\]
Note that while $x^2/2 + \qone_{\al,\la}(x)$ is non-negative and strictly convex, this is probable not satisfied for $\qone_{\al,\la}$. If $\sqrt{\al}/(d \ka_\la) < 1$, then $s_{\al,\la}$ can be non-positive everywhere. If $\sqrt{\al}/(d \ka_\la) \geq  1$, then $\qone_{\al,\la}(x)$ has to be non-negative for $\abs{x} \leq 2 c \al/\ka_\la$, but can be negative elsewhere.
Assumption \ref{ass:A3} and \ref{ass:A2} essentially mean that $(\qone_{\al,\la}(x))_{\al,\la}$ is bounded from below by $-(K-1)x^2/2$ and for a fixed $\ka_\la$ the family $(\qone_{\al,\la}(x))_{\al \in (0,\tilde \al)}$ is bounded from above for all $x \in \R$.
\end{remark}

\begin{proposition}[Convergence of $\Mo^+_{\kao} \circ \, \Phi_{\al,\kao}$]\label{prop:conv}
Let $(\ph_{\al})_{\al > 0}$ be weakly convex regularizing filter such that Assumption \ref{Ass} is satisfied. Suppose $\al_n, \al > 0$ and $z, z^n \in \ell^2(\La)$ with  $\norm{z-z^n}_2 \leq \delta_n \rightarrow 0$.
\begin{itemize}[leftmargin=2em]
\item  \textit{Existence:} $\dom(\Mo^+_{\kao} \circ \, \Phi_{\al,\kao} )= \ell^2(\La) $.
\item  \textit{Stability:} $\Mo^+_{\kao} \circ \Phi_{\al,\kao} (z^n) \rightharpoonup \Mo^+_{\kao} \circ \, \Phi_{\al,\kao} (z)$.
\item \textit{Convergence:} Assume $z\in \ran(\Mo_{\kao}) \cap \ran(\Phi_{\al,\kao})$ for all $\al \in (0,\tilde \al)$ and some $\tilde{\al}> 0$. Let $\al_n \rightarrow 0$ such that $\al_n \gtrsim \delta_n^2$. Then $\Mo^+_{\kao} \circ \Phi_{\al_n,\kao} (z^n) \rightharpoonup \Mo^+_{\kao} z$.
\end{itemize}
\end{proposition}

\begin{proof}
The proof can be found in the Appendix. 
\end{proof}

By the continuity of the frame operators and the above convergence results we can formulate that under assumption \ref{Ass} the non-linear filtered DFD is a convergent regularization method. 

\begin{theorem}[Convergence of $\Fc_\al$]
Let $(\ph_{\al})_{\al > 0}$ be a weakly convex regularizing filter satisfying Assumption \ref{Ass}, $(\delta_k)_k$, $ (\al_k)_k$ two null sequences such that  $\al_k \gtrsim \delta_k^2$ and let $y^k, y \in \Y$ such that $\snorm{y^k -y} \leq \delta_k$. 
\begin{itemize}[leftmargin=2em]
\item  \textit{Existence:} $\dom(\Fc_\al )= \ell^2(\La)$.
\item  \textit{Stability:}  Let $\al > 0$ be fixed. Then, $\Fc_{\al} (y^k) \rightharpoonup \Fc_\al$.
\item \textit{Convergence:} Suppose $y \in \ran(\Ao)$ such that $\To_{\vo}^*(y) \in \ran(\Phi_{\al,\kao})$ for all $\al \in (0,\tilde\al)$ for some $\tilde\al > 0$. Then, $\Fc_{\al_k}(y^k) \rightharpoonup \Ao^+ y$.
\end{itemize}
\end{theorem}

\begin{proof}
Existence and Stability follow directly from Proposition \ref{prop:conv} and the continuity of $\To_{\bar \uo}$ and $\To_{\vo^*}$. 
For the convergence take $x^+ \in \ker(\Ao)^\perp$ such that $\Ao x^+ = y$, then 
\[ \norm{\To_\vo^*(y_k)-\Mo_\kao \To_\uo^* x^+} = \norm{\To_\vo^*(y_k)-\To_\vo^* \Ao x^+} = \norm{\To_\vo^*(y_k)-\To_\vo^*(y)}\leq \norm{\To_\vo}\delta_k. \]
Define $z^k = \Mo_\kao^+ \circ \Phi_{\al_k, \kao} \circ \To_\vo^* (y^k)$. By Proposition \ref{prop:conv}, we have $z^k \rightharpoonup \Mo_\kao^+ \Mo_\kao \To_\uo^*x^+ = \To_\uo^*x^+$.
For any $e \in \X$, we have
\[ \inner{e}{\To_{\bar\uo} z^k - x^+ } = \inner{e}{\To_{\bar\uo} z^k - \To_{\bar\uo}\To_\uo^*x^+ } = \inner{\To^*_{\bar\uo} e}{z^k - \To_\uo^*x^+ } \rightarrow 0,\]
so $x^k = \To_{\bar\uo} z^k \rightharpoonup x^+$.
\end{proof}

\subsection{Convergence Rate}

In the following, we introduce the absolute symmetric Bregman distance and investigate the rate of convergence of the non-linear DFD with respect to that measure. While these estimates are rather abstract, in some cases, this leads directly to strong convergence and rates in the norm topology.

\subsubsection{Absolute Symmetric Bregman Distance}

The next goal is to analyze the rate of convergence. When dealing with weak convergence, a convenient strategy for obtaining qualitative estimates is to use Bregman distances. These distances are widely employed to describe the quality of convergence in convex optimization algorithms or regularization methods for inverse problems, as seen in, for example, \cite{Burger2004, Burger2016, Abascal2018, Ebner2023}.

Originally the Bregman distance describes the distance between two points $x$ and $y$ with respect to a proper, convex and lower semi-continuous function $\Qc$ defined as 
\[D_\Qc^{p}(x,y) = \Qc(x)-\Qc(y) - \inner{p}{x-y},\]
where $p \in \partial \Qc(y)$. Note that if $\Qc$ is additionally Gateaux differentiable, the only sub-gradient is just $p=\nabla \Qc(y)$. 
However, this mapping is not a distance; it lacks symmetry, and in the case of Gateaux differentiable non-convex functionals, it is well defined but may be negative.
Therefore, we rely on a generalized version of the Bregman distance, known as the absolute symmetric Bregman distance, defined by
\[  D_\Qc (x,y) \coloneqq \abs{D_\Qc^{\nabla \Qc(y)}(x,y) + D_\Qc^{\nabla \Qc(x)}(y,x)} = \abs{\inner{\nabla \Qc(x)-\nabla \Qc(y)}{x-y}} \]
for a differantiable but non-convex functional $\Qc \colon \dom(\Qc) \subset \X \rightarrow \R $. 
While this symmetrification of the Bregman distance is often used (first in \cite{burger2007error}), the positivization was introduced recently in \cite{OBMANN2024, obmann2023convergence} to address the lack of positivity when using non-convex functionals.

It's important to note that there are other generalizations of the Bregman distance for non-convex functionals; for example, in \cite{Grasmair2010}, the author extends the Bregman distance to locally convex functionals.

\subsubsection{Neighbouring Stationary Regularization}

In classical variational regularization the regularizing term depends linear on the regularization parameter, as a stationary $\kao$-regularizer $\al \Qc$, and rates are analysed with respect to the regularizer $\Qc$. In our case the regularization parameter is given implicitly and the convergence behaviour of $\Rc_\al$ as $\al \downarrow 0$ is uncontrolled so far. 
However, if $\Rc_\al$ lies in certain bounded areas which can be described by stationary regularizing terms then we obtain a similar behaviour. 

The idea is to find a family of functions $(q_\la)_\la$ such that $s_{\al,\la}$ is close to $\al q_\la$ for every $\la \in \La$. By close we mean that for every $\la \in \La$ and every $x \in \R$ holds 
\[\abs{s_{\al, \la} (x) - \al q_\la(x)} \leq \frac{L\al}{\ka_\la} x^2\]
for some fixed constant $L > 0$.
Given such a family $(q_\la)_\la$ one may define the functional $\Qc(x) \colon \ell^2(\La) \rightarrow \R \cup \{\infty\} \colon (x_\la)_\la \mapsto \sum_{\la \in \La} q_\la(\ka_\la x_\la)$. Then, 
\[ \abs{\Rc_\al (x)- \al \Qc(x)} \leq \sum_{\la \in \La} \abs{s_{\al,\la}(\ka_\la x_\la) - \al q_\la(\ka_\la x_\la)} \leq L \al \max_{\la \in \La} \ka_\la \norm{x}^2 \]
for all $x = (x_\la)_\la \in \ell^2(\La)$.

\begin{definition}[Neighbouring stationary $\kao$-regularizer]\label{def:Neigh}
For some $L > 0$ and some $\tilde \al > 0$ small enough consider a family $(q_\la)_{\la \in \La}$ of functions $q_\la \colon \R \rightarrow \R$ with the following properties:
\begin{enumerate}[label=(Q\arabic*), leftmargin =2.5em]
\item \label{Ass:Q1} $q_\la$ is continuously differentiable, 
\item \label{Ass:Q2} $q'_\la(0)=0$, 
\item \label{Ass:Q3} for $x < y \colon (q'_\la(y) - q'_\la(x))/(y-x) \geq 2L/\ka_\la - 1/\tilde{\al}$. 
\end{enumerate}
Define $\Qc \colon \dom(\Qc) \rightarrow \ell^2 \colon x \mapsto \sum_{\la \in \La} q_\la(\ka_\la x_\la)$ where $\dom(\Qc) \coloneqq \{x \in \ell^2(\La) \mid ((\ka_\la x_\la)^2/(2\tilde \al) + q_{\la}(\ka_\la x))_\la \in \ell^1(\La)  \}$. 
Let $(\ph_\al)_{\al >0}$ be a regularizing filter. 
We call $\al\Qc$ a neighbouring stationary $\kao$-regularizer to $\Mo^+_{\kao} \circ \Phi_{\al,\kao}$ if 
\begin{align}\label{phi_inte}
\abs{\ph_{\al}(\ka_\la,x)} \in \left[ \abs{\prox_{\al \left( q_\la + L/\ka_\la(\cdot)^2\right)}(x)}, \abs{\prox_{\al \left( q_\la - L/\ka_\la(\cdot)^2\right)}(x) }\right]
\end{align}  
holds for all $x\in \R$, $\la \in \La$ and $\al \in (0,\tilde \al)$.
\end{definition}

Assumption \ref{Ass:Q3} implies that $\id + \al(q'_\la \pm 2L/\ka_\la (\cdot))$ is bijective and increasing for all $\al \in (0,\tilde \al)$. Then, one can easily verify that  $\al \left( q_\la \pm L/\ka_\la(\cdot)^2\right)$ is weakly strictly convex and weakly coercive, hence its proximity operator is well-defined. 
Note that, by Assumption \ref{Ass:Q3}, $q'_\la$ has to be  strictly increasing and hence $q_\la$ is convex if $\ka_\la \leq 2 L \tilde \al $. Since $\sup_{\la \in \La} \ka_\la < \infty$ this holds for all $\la$ except for finitely many.

\begin{proposition}[Differentiability of $\Qc$]
Define 
\[\dom(\Qc') \coloneqq \{x \in \ell^2(\La) \mid (\ka_\la q_{\la}'(\ka_\la x_\la))_\la \in \ell^2(\La) \}.\]
We have that $\Qc$ is Gateaux differentiable on $\dom(\Qc)^\circ \, \cap \, \dom(\Qc')$ with gradient $\nabla \Qc(x) = (\ka_\la q_{\la}'(\ka_\la x_\la))_\la$.
\end{proposition}

\begin{proof}
Fix $L, \tilde \al > 0$.
Consider $r_\la(x) = (\ka_\la x)^2/(2\tilde \al) + q_{\la}(\ka_\la x)$. Then the proof is similar to the proof of Proposition \ref{prop:diff_R}. We only have to show that $r_\la'$ is increasing. Let $x < y$ then we have
\begin{align*}
    r_\la'(y)-r_\la'(x) &= y \ka_\la^2/\tilde \al  + \ka_\la q'_{\la}(\ka_\la y) - x \ka_\la^2/\tilde \al - \ka_\la q'_{\la}(\ka_\la x) \\
    &=  (y-x) \ka_\la^2/\tilde \al  + \ka_\la (q'_{\la}(\ka_\la y) - q'_{\la}(\ka_\la x))\\
    & \geq (y-x) (\ka_\la^2/\tilde \al + \ka_\la^2 (2L/\ka_\la -1/\tilde \al))  =  (y-x) 2L\ka_\la > 0. 
\end{align*}
\end{proof}

We show rates of convergence in the symmetric Bregman-distance with respect to $\Qc$.

\begin{theorem}[Rates in Symmetric Bregman-distance] \label{rate}
Let $(\ph_{\al})_{\al > 0}$ be a weakly convex regularizing filter which satisfies Assumption \ref{Ass} and assume that $\Mo^+_{\kao} \circ \Phi_{\al,\kao}$ has a neighbouring stationary $\kao$-regularizer $\al\Qc$.
Suppose $z\in \ran(\Mo_{\kao}) \cap \ran(\Phi_{\al,\kao})$ for all $\al \in (0,\tilde \al)$ and $z^n \in \ell^2(\La)$ with  $\norm{z-z^n}_2 \leq \delta_n \rightarrow 0$.
Let $\al_n \rightarrow 0$ and define $x^n \coloneqq \Mo^+_{\kao} \circ \Phi_{\al_n,\kao} (z^n)$.
Further, suppose that $x^n , \Mo_{\kao}^+ z \in \dom(\Qc)^\circ \, \cap \, \dom(\Qc')$ for all $n \in \N$ and $(q_\la'(z_\la))_\la \in \ell^2$. Then, it holds
\[ D_\Qc(x^n,\Mo_{\kao}^+ z) \leq \frac{1}{2} \frac{\delta_n^2}{\al_n} + C \delta_n + C^2 \al_n \]
for a constant $C > 0$. 
\end{theorem}

\begin{proof}
First note that by Lemma \ref{lem_main} we have $x^n \in \dom(\Rc_{\al_n})^\circ \cap \dom(\Rc_{\al_n}')$ and hence $\nabla \Rc_\al(x^n)$ exists and is equal to $(\ka_\la s_{\al,\la}'(\ka_\la x_\la^n))_{\la \in \La}$.
By Lemma \ref{lem_main} we have that
\[ \nabla \Rc_{\al_n}(x^n) = - \Mo_{\kao} (\Mo_{\kao} x^n - z^n). \]
We consider 
\begin{align*}
\al_n \abs{\inner{\nabla \Qc(x^n) - \nabla \Qc(\Mo_{\kao}^+ z)}{x^n -\Mo_{\kao}^+ z} } & = \al_n \abs{\inner{\nabla \Qc(x^n) - \frac{\nabla \Rc_{\al_n}(x^n)}{\al_n}}{x^n -\Mo_{\kao}^+ z}} \\
& + \al_n \abs{\inner{\frac{\nabla \Rc_{\al_n}(x^n)}{\al_n}}{x^n -\Mo_{\kao}^+ z} } \\
& + \al_n \abs{\inner{\nabla \Qc(\Mo_{\kao}^+ z)}{\Mo_{\kao}^+ z - x^n}}
\end{align*}
and examine every term separately.

We start with the second term, which can be estimated straight forward with the sub-gradient $\nabla \frac{1}{2} \norm{\Mo_\kao (\cdot) - z^n}^2 = \Mo_{\kao} (\Mo_{\kao} (\cdot) - z^n)$.
 
\begin{align*}
\al_n \abs{\inner{\frac{\nabla \Rc_{\al_n}(x^n)}{\al_n}}{x^n -\Mo_{\kao}^+ z}} & = \abs{\inner{\Mo_{\kao} (\Mo_{\kao} x^n - z^n)}{\Mo_{\kao}^+ z - x^n }} \\
& \leq \frac{1}{2} \norm{\Mo_\kao \Mo_{\kao}^+ z - z^n}^2 - \frac{1}{2} \norm{\Mo_\kao x^n - z^n}^2 \\
& \leq \frac{1}{2} \delta_n^2 - \frac{1}{2} \norm{\Mo_\kao x^n - z^n}^2.
\end{align*}

In the third term, since $\nabla \Qc(\Mo_{\kao}^+ z) = (\ka_\la q_\la'(z_\la))_\la \in \im(\Mo_{\kao})$, there exists an $\eta \in \ell^2$ such that $\nabla \Qc(\Mo_{\kao}^+ z) = \Mo_\kao \eta$ and we get
\begin{align*}
\abs{\inner{\nabla \Qc(\Mo_{\kao}^+ z)}{\Mo_{\kao}^+ z - x^n}} & = \abs{\inner{\eta}{\Mo_\kao (\Mo_{\kao}^+ z - x^n)}} \\
& \leq \norm{\eta} \left( \norm{z- z^n} + \norm{ \Mo_\kao x^n - z^n} \right)\\
& \leq \norm{\eta} \left( \delta_n + \norm{ \Mo_\kao x^n - z^n} \right).
\end{align*}

In first term there is a bit more work to do. First we show that $\nabla \Qc(x^n)$ is in the range of $\Mo_\kao$. For every $\la \in \La$ define 
\begin{align*}
\eta^n_\la \coloneqq \frac{1}{\ka_\la} \left( \frac{s_{\al_n,\la}(\ka_\la \cdot)'(x_\la^n)}{\al_n}- q_\la(\ka_\la \cdot)' (x_\la^n) \right) = \frac{\ph_{\al_n}(\ka_\la,\cdot)^{-1}(\ka_\la x_\la^n) - \ka_\la x_\la^n}{\al_n}- q_\la'(\ka_\la x_\la^n) .
\end{align*}
By \eqref{phi_inte} holds
\[ 
\frac{\ph_{\al_n}(\ka_\la,\cdot)^{-1}(x) - x}{\al_n} \in \left[ q'_\la(x) - \frac{2L}{\ka_\la}\abs{x}, q'_\la(x) + \frac{2L}{\ka_\la}\abs{x} \right]
\]
for all $x \in \R$. 
Therefore, we have $\abs{\eta^n_\la} \leq 2L \abs{x_\la^n}$ and as a consequence $\norm{(\eta_\la^n)_{\la \in \La}} = \norm{\eta^n} \leq 2L \norm{x^n}$. This shows that $\eta^n \in \ell^2(\La)$ for every $n \in \N$ and by Proposition \ref{prop:conv} the sequence $(\eta^n)_{n \in \N}$ is bounded.

Thus, we have $\nabla \Qc(x^n) -\frac{\nabla \Rc_{\al_n}(x^n)}{\al_n} = -\Mo_\kao \eta^n$ and similar to above we get
\begin{align*}
\abs{\inner{\nabla \Qc(x^n)-\frac{\nabla \Rc_{\al_n}(x^n)}{\al_n}}{x^n - \Mo_{\kao}^+ z}} \leq \norm{\eta^n} \left( \delta_n + \norm{ \Mo_\kao x^n - z^n} \right).
\end{align*}

Putting all together we have 
\begin{multline*}
\al_n \abs{\inner{\nabla \Qc(x^n) - \nabla \Qc(\Mo_{\kao}^+ z)}{x^n -\Mo_{\kao}^+ z}} \\\ \leq \frac{1}{2} \delta_n^2 - \frac{1}{2} \norm{\Mo_\kao x^n - z^n}^2 + \al_n \left(\norm{\eta} + \norm{\eta^n}\right) \left( \delta_n + \norm{ \Mo_\kao x^n - z^n} \right)
\end{multline*}

and by the inequality of arithmetic and geometric means follows 
\begin{multline*}
\frac{1}{2} \norm{\Mo_\kao x^n - z^n}^2 + \al_n \abs{\inner{\nabla \Qc(x^n) - \nabla \Qc(\Mo_{\kao}^+ z)}{x^n -\Mo_{\kao}^+ z}} \\ \leq \frac{1}{2} \delta_n^2 + \al_n \left(\norm{\eta} + \norm{\eta^n}\right) \left( \delta_n + \norm{ \Mo_\kao x^n - z^n} \right) \\ 
\leq \frac{1}{2} \delta_n^2 + \al_n C \delta_n + C^2 \al_n^2 + \frac{1}{4}\norm{ \Mo_\kao x^n - z^n}^2, 
\end{multline*}
which shows the statement.
\end{proof}

\begin{remark}[Impact of the Neighbouring Condition \eqref{phi_inte} to $\qone_{\al,\la}$]
Suppose Assumption \ref{Ass} and condition \eqref{phi_inte} holds. Then, by the definition of the general proximity operator and integration we can reduce \eqref{phi_inte} to 
\[\frac{x^2}{2} + \qone_{\al,\la} (x)  \in \left[ \al q_\la (x) + (1- \frac{2 L \al}{\ka_\la}) \frac{x^2}{2}, \al q_\la(x) + (1+ \frac{2 L \al}{\ka_\la}) \frac{x^2}{2} \right]. \]
Combing this with Remark \ref{rem:Ass_A} we know that 
\[
\al \left(q_\la (x) + \frac{L x^2}{\ka_\la} \right) \geq \qone_{\al,\la}(x) \geq
\begin{cases}
    \max\{\left(\frac{\sqrt{\al}}{d \ka_\la} -1 \right) \frac{x^2}{2}, \al \left( q_\la (x) - \frac{L x^2}{\ka_\la}\right) \} & \abs{x} < 2 c d \sqrt{\al} \\
     \max\{\frac{c \al}{\ka_\la}\abs{x} - \frac{x^2}{2}, \al \left( q_\la (x) - \frac{L x^2}{\ka_\la} \right) \} & \abs{x} \geq 2 c d \sqrt{\al}. 
\end{cases}
\]

\end{remark}

\begin{example}[Smallest $\Qc$]
For a fixed $L>0$ and $\tilde \al > 0$ consider the functions $q_\la(x) = (L/\ka_\la-1/(2\tilde \al)) x^2$ for all $\la \in \La$, which is the smallest possible choice concerning Definition \ref{def:Neigh}. 
Then, the filters have to satisfy
\[ 
\begin{rcases}
    \min\{ \frac{1}{1-\al/\tilde \al}, \frac{d \ka_\la}{\sqrt{\al}}\} \abs{x} & \abs{x} \leq c \al/\ka_\la \\
     \frac{\abs{x}}{1-\al/\tilde \al} & \abs{x} \geq c \al/\ka_\la
\end{rcases}
\geq (\sign{x} )\ph_{\al}(\ka_\la, x) \geq \frac{\abs{x}}{1-\al/\tilde \al + 2L \al/\ka_\la} \]
and in the following
\[
\al \left(\frac{2 L}{\ka_\la} -\frac{1}{\tilde \al} \right) \frac{x^2}{2} \geq \qone_{\al,\la}(x) \geq
\begin{cases}
    \max\{\left(\frac{\sqrt{\al}}{d \ka_\la} -1 \right) \frac{x^2}{2}, -\frac{\al x^2}{2\tilde\al}\} & \abs{x} < 2 c d \sqrt{\al} \\
     \max\{\frac{c \al}{\ka_\la}\abs{x} - \frac{x^2}{2}, -\frac{\al x^2}{2\tilde\al} \} & \abs{x} \geq 2 c d \sqrt{\al}.
\end{cases}
\]
 
In other words, $q_\la$ defines a region where $\ph_{\al}(\ka_\la, \cdot)$ and $\qone_{\al,\la}$ should reside, respectively. The upper part of Figure \ref{fig:area} graphically illustrates the constraints when assumption \ref{ass:A1} partly dominates condition \eqref{phi_inte}. In the lower part of Figure \ref{fig:area}, the scenario where assumption \ref{ass:A1} contributes nothing beyond condition \eqref{phi_inte} is depicted. This is primarily observed in subsection \ref{subsec:norm}, except for a negligible quantity.
Note that for small $\ka_\la$, the region of the filter functions approaches zero, and for small $\al$, it tends towards the identity.
\end{example}

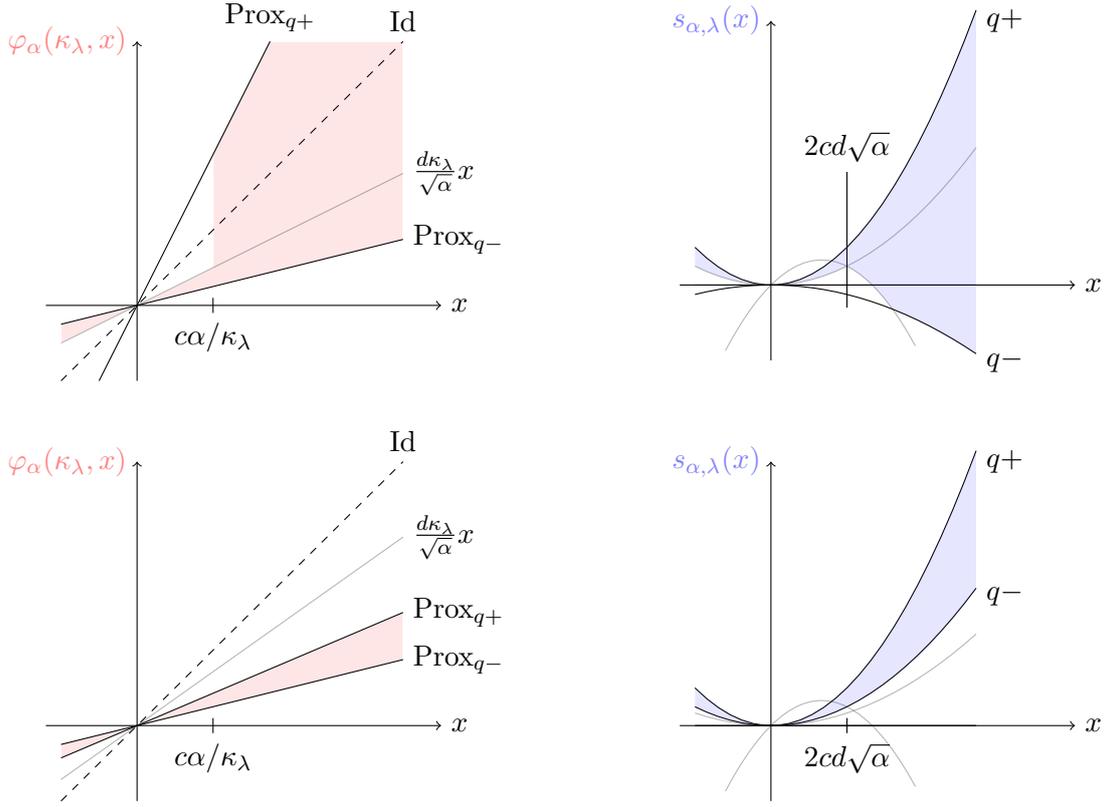
\begin{figure}[htb!]
\centering
\begin{tikzpicture}
\draw[white, fill=red!10] (0,0) -- (1,0.5) -- (1,2) -- (1.75,3.5) -- (3.5,3.5) -- (3.5, 0.875);
\draw[white, fill=red!10] (0,0) -- (-1, -1/2) -- (-1, -1/4);
\draw[->] (-1.2, 0) -- (4, 0) node[right] {$x$};
\draw[->] (0, -1) -- (0, 3.5) node[red!50, left] {$\ph_{\al}(\ka_\la,x)$};
\draw[black, dashed] (-1, -1) -- (3.5, 3.5);
\node at (3.5,3.5) [above, black] {$\id$};
\draw[black] (-1/2, -1) -- (1.75, 3.5);
\node at (1.75, 3.5) [above, black] {$\prox_{q+}$};
\draw[black, opacity=0.3] (-1, -1/2) -- (3.5, 1.75);
\node at (3.5, 1.75) [right, black] {$\frac{d\ka_\la}{\sqrt{\al}} x $};
\draw[black] (-1, -1/4) -- (3.5, 0.875);
\node at (3.5,0.875) [right, black] {$\prox_{q-}$};
\draw[black] (1,-0.1) -- (1,0.1);
\node at (1,-0.1) [below, black] {$c \al/\ka_\la$};
\end{tikzpicture}
\hfill
\begin{tikzpicture}
\draw[->] (-1.2, 0) -- (4, 0) node[right] {$x$};
\draw[->] (0, -1) -- (0, 3.5) node[blue!50, left] {$\qone_{\al,\la}(x)$};
\draw (-1, 0) -- (2.7, 0);
\node at (2.7,3.5) [right, black] {$q+$};
\draw[name path=A, domain=-1:2.7,variable=\x,  black] plot ({\x}, {1/2*\x*\x});
\draw[name path=B, domain=-1:2.7,variable=\x,  black, opacity=0.3] plot ({\x}, {1/4*\x*\x});
\draw[name path=C, domain=1:8/5,variable=\x,  black, opacity=0] plot ({\x}, {1/4*\x*\x});
\draw[name path=D, domain=8/5:2.7,variable=\x,  black, opacity=0] plot ({\x}, {1/4*\x*\x});
\tikzfillbetween[of=A and B, on layer=ft]{blue, opacity=0.1};
\draw[domain=-1:2.7,variable=\x,  black] plot ({\x}, {-\x*\x/8});
\draw[name path=F, domain=8/5:2.7,variable=\x,  black,  opacity=0] plot ({\x}, {-\x*\x/8});
\node at (2.7,-1) [right, black] {$q-$};
\draw[domain=-0.6:1.9,variable=\x,  black,  opacity=0.3] plot ({\x}, {\x -3/4*\x*\x});
\draw[name path=E, domain=1:8/5,variable=\x,  black, opacity=0] plot ({\x}, {\x -3/4*\x*\x});
\tikzfillbetween[of=C and E, on layer=ft]{blue, opacity=0.1};
\tikzfillbetween[of=D and F, on layer=ft]{blue, opacity=0.1};
\draw[black] (1,-0.3) -- (1,1.5);
\node at (1,1.5) [above, black] {$2 c d \sqrt{\al}$};
\end{tikzpicture}
\hfill \\[0.5cm]
\begin{tikzpicture}
\draw[white, fill=red!10] (0,0) -- (3.5,1.5) -- (3.5, 0.875);
\draw[white, fill=red!10] (0,0) -- (-1, -1.5/3.5) -- (-1, -1/4);
\draw[->] (-1.2, 0) -- (4, 0) node[right] {$x$};
\draw[->] (0, -1) -- (0, 3.5) node[red!50, left] {$\ph_{\al}(\ka_\la,x)$};
\draw[black, dashed] (-1, -1) -- (3.5, 3.5);
\node at (3.5,3.5) [above, black] {$\id$};
\draw[black] (-1, -1.5/3.5) -- (3.5, 1.5);
\node at (3.5,1.5) [right, black] {$\prox_{q+}$};
\draw[black, opacity=0.3] (-1, -2.5/3.5) -- (3.5, 2.5);
\node at (3.5,2.5) [right, black] {$\frac{d\ka_\la}{\sqrt{\al}} x $};
\draw[black] (-1, -1/4) -- (3.5, 0.875);
\node at (3.5,0.875) [right, black] {$\prox_{q-}$};
\draw[black] (1,-0.1) -- (1,0.1);
\node at (1,-0.1) [below, black] {$c \al/\ka_\la$};
\end{tikzpicture}
\hfill
\begin{tikzpicture}
\draw[->] (-1.2, 0) -- (4, 0) node[right] {$x$};
\draw[->] (0, -1) -- (0, 3.5) node[blue!50, left] {$\qone_{\al,\la}(x)$};
\draw (-1, 0) -- (2.7, 0);
\node at (2.7,3.5) [right, black] {$q+$};
\draw[name path=A, domain=-1:2.7,variable=\x,  black] plot ({\x}, {1/2*\x*\x});
\node at (2.7,1.75) [right, black] {$q-$};
\draw[name path=B, domain=-1:2.7,variable=\x,  black] plot ({\x}, {1/4*\x*\x});
\tikzfillbetween[of=A and B, on layer=ft]{blue, opacity=0.1};
\draw[domain=-1:2.7,variable=\x,  black,  opacity=0.3] plot ({\x}, {1/6*\x*\x});
\draw[domain=-0.6:1.9,variable=\x,  black,  opacity=0.3] plot ({\x}, {\x -3/4*\x*\x});
\draw[black] (1,-0.1) -- (1,0.1);
\node at (1,-0.1) [below, black] {$2 c d \sqrt{\al}$};
\end{tikzpicture}
\caption{This plot shows a graphically illustration of the neighbouring condition \eqref{phi_inte} and assumption \ref{ass:A1}. The red area is where $\ph_{\al}(\ka_\la,\cdot)$ should belong and the blue area is for the corresponding functional $\qone_{\al,\la}$. Here $\prox_{q+}$ refers to the right bound of condition \eqref{phi_inte}, namely the function $\prox_{\al ( q_\la - L(\cdot)^2)/\ka_\la }$, and $\prox_{q-}$ refers to the left bound. The upper part of  illustrates when assumption \ref{ass:A1} partly dominates condition \eqref{phi_inte}. In the lower part, the scenario where assumption \ref{ass:A1} contributes nothing beyond condition \eqref{phi_inte} is depicted. }\label{fig:area}
\end{figure}

\begin{theorem}[Convergence Rate with respect to the Symmetric Bregman Distance]
Let $(\ph_{\al})_{\al > 0}$ be a weakly convex regularizing filter which satisfies Assumption \ref{Ass}. 
Assume there exists a neighbouring stationary $\kao$-regularizer $\al\Qc$ to $\Mo^+_{\kao} \circ \Phi_{\al,\kao}$.
Suppose $y \in \ran(\Ao)$ such that $\To_{\vo}^*(y) \in \ran(\Phi_{\al,\kao})$ for all $\al \in (0,\tilde\al)$.
Let $(\delta_k)_k$, $ (\al_k)_k$ be two null sequences such that  $\al_k \cong \delta_k$ and let $y^k \in \Y$ such that $\snorm{y^k -y} \leq \delta_k$. 
Further, suppose that $\To_\uo^* (\Fc_{\al_k} (y^k)), \To_\uo^* \Ao^+ y \in \dom(\Qc)^\circ \, \cap \, \dom(\Qc')$ for all $k \in \N$ and $(q'_\la(\inner{u_\la}{\Ao^+ y}))_\la \in \ell^2(\La)$. 
Then, it holds
\[ 
D_{\Qc \circ \To^*_\vo} (\Fc_{\al_k}(y^k), \Ao^+ y)\leq C \delta_k \]
for a constant $C > 0$. 
\end{theorem}

\begin{proof}
Note that the frame operators are linear and bounded, and therefore also Frechet differentiable. Then, by the chain rule and Theorem \ref{rate}, we have 
\begin{multline*}
\inner{\nabla ( \Qc \circ \To_\uo^*) (\Fc_{\al_k}(y^k)) - \nabla ( \Qc \circ \To_\uo^*) (x^+)}{\Fc_{\al_k}(y^k) - \Ao^+ y}  \\
= \inner{\nabla \Qc(\To_\uo^* \Fc_{\al_k}(y^k)) - \nabla \Qc (\To_\uo^* \Ao^+ y)}{\To_\uo^* ( \Fc_{\al_k}(y^k) - \Ao^+ y)}
\\ \leq \frac{\delta_n^2}{2\al_n} + C \delta_n + C^2 \al_n
\cong C \delta_n.
\end{multline*}
\end{proof}

\subsubsection{Estimates in the Norm Topology} \label{subsec:norm}

Consider $q_\la(x)= (\max_\la \ka_\la) L x^2/ \ka_\la^2$. Then, $\dom(\Qc)^\circ \, \cap \, \dom(\Qc') = \ell^2(\La)$ and $\nabla \Qc(x)= (\max_\la \ka_\la) 2 L x$ for all $x \in \ell^2(\La)$.
By this, we get a square root rate of convergence of the normed error:
\[\norm{\Ao^+ y- \Fc_{\al_k}(y^k)}^2 
 \leq \frac{\norm{(\To^*_\vo)^{-1}}^2}{(\max_\la \ka_\la) 2 L} D_{\Qc \circ \To^*_\vo} (\Fc_{\al_k}(y^k), \Ao^+ y)\leq C \delta_k. \]
The filters have to satisfy for $\abs{x} \leq c \al/\ka_\la \colon$
\begin{align*}
    \abs{\ph_{\al}(\ka_\la, x)} &\in \left[\frac{\abs{x}}{1+\left(\frac{\max_\la \ka_\la}{\ka_\la} + 1\right) \frac{2L \al}{\ka_\la}}, \min\left\{ \frac{1}{1+\left(\frac{\max_\la \ka_\la}{\ka_\la} -1\right) \frac{2L \al}{\ka_\la}} , \frac{d \ka_\la}{\sqrt{\al}}\right\} \abs{x} \right],
\end{align*}
otherwise, for $\abs{x} \geq c \al/\ka_\la \colon$
\begin{align*}
     \abs{\ph_{\al}(\ka_\la, x)} &\in \left[ \frac{\abs{x}}{1+\left(\frac{\max_\la \ka_\la}{\ka_\la} +1\right) \frac{2L \al}{\ka_\la}}, \frac{\abs{x}}{1+\left(\frac{\max_\la \ka_\la}{\ka_\la} -1\right) \frac{2L \al}{\ka_\la}} \right]
\end{align*}
and in the following, for $\abs{x} < 2 c d \sqrt{\al} \colon$
\begin{align*}
     \qone_{\al,\la}(x) &\in \left[\max\{\left(\frac{\sqrt{\al}}{d \ka_\la} -1 \right) \frac{x^2}{2}, \al \left(\frac{\max_\la \ka_\la}{\ka_\la} - 1 \right) \frac{L x^2}{\ka_\la}\}, \al \left(\frac{\max_\la \ka_\la}{\ka_\la} + 1 \right) \frac{L x^2}{\ka_\la} \right],
\end{align*}
and for $\abs{x} \geq 2 c d \sqrt{\al} \colon $
\begin{align*}
    \qone_{\al,\la}(x) &\in \left[ \max\{\frac{c \al}{\ka_\la}\abs{x} - \frac{x^2}{2}, \al \left(\frac{\max_\la \ka_\la}{\ka_\la} - 1 \right) \frac{L x^2}{\ka_\la} \}, \al \left(\frac{\max_\la \ka_\la}{\ka_\la} + 1 \right) \frac{L x^2}{\ka_\la} \right].
\end{align*}

Actually, Assumption \ref{ass:A1} in this case only plays a role for relative large $\al$, depending on the constants $d, L$ and $\max_\la \ka_\la$, and large $\ka_\la$. 

\begin{remark} 
It is important to note that condition \eqref{phi_inte} implies that 
\[\Qc(x) - L \snorm{\Mo_{\sqrt{\kao}}x}^2 \leq \frac{\Rc_\al(x)}{\al} \leq \Qc(x) + L \snorm{\Mo_{\sqrt{\kao}}x}^2, \]
which is reminiscent of the key condition in the work \cite{hauptmann2023convergent} for proving strong convergence of their regularization method. The authors attribute the Plug-and-Play method with a linear denoiser to variational regularization with a non-stationary regularizing term $\mathcal{J}_\al(x)$. Bounding this regularizing term between $c\norm{x}^2/2$ and $C\norm{x}^2/2$, where $0 < c \leq C < \infty$, appears to be a crucial argument for proving that the regularized solution strongly converges to an exact solution.
Defining $\Qc$ as above leads to: 
\[ L \sum_{\la \in \La} (\max \ka_\la-\ka_\la)x_\la^2 \leq \frac{\Rc_\al(x)}{\al} \leq L \sum_{\la \in \La} (\max \ka_\la+\ka_\la)x_\la^2, \]
which also indicates strong convergence and provides qualitative estimates with respect to the norm topology.
\end{remark}

\subsection{Discussion}
For filters learned on sinograms with additive Gaussian noise, we numerically investigated if they satisfy the theoretically derived properties to yield a convergent regularization method. To this end, we set $\alpha \coloneqq \sigma$, where $\sigma$ is the standard deviation of the Gaussian distribution, and defined $\ph_\alpha$ as the learned filters for the corresponding noise level. Table \ref{tab:properites} shows how well this learned filters satisfied the different assumption. It turned out that the filters satisfied the properties of weakly convex regularizing filters given in Definition \ref{def:filter} quite well. For all different noise levels, the filters were strictly increasing \ref{def:filter2}, and by construction, they are continuous, which further implies bijectivity \ref{def:filter1}. Also the values of $\ph_\alpha(k_\lambda,0)$ were very close to zero for all $\la$ and all noise levels \ref{def:filter3}. For smaller and smaller noise levels, the filters tend to converge toward the identity as $\sum_\la\norm{\ph_\alpha(\kappa_\la,\cdot)-\id}_{L^2}$ gets smaller for $\alpha$ getting smaller \ref{def:filter4}. Also, Assumption \ref{Ass} was satisfied by the learned filters. We observed that $\underset{\la,0\neq|x|<\alpha/\kappa}{\max} \left\{\frac{|\ph_\alpha(\kappa_\la, x)| \sqrt{\alpha}}{|x| \kappa_\la} \right\}$ does not grow for smaller $\al$ and hence stays bounded by a constant $d$ \ref{ass:A1}.  $\abs{\ph_\al(\ka_\la,x)}$ is easily bounded by $2\abs{x}$ for all $\la$, $\al$, and $x$ \ref{ass:A3}. Finally, choosing $g_\kappa(x) = \kappa^2 x$, we got that $\underset{\la,x \neq 0}{\min} \left\{\frac{|\ph_\alpha(\kappa_\lambda, x)|} {g_{\kappa_\lambda} (|x|)} \right\}$ stayed bigger than one confirming \ref{ass:A2}.  
\begin{table*}[!t]
\caption{Theoretically derived properties that the learned filters must satisfy in order to be a convergent regularization method checked numerically. All properties and assumptions are fulfilled with sufficient accuracy.}

\label{tab:properites}
\centering
\resizebox{\textwidth}{!}{%
\begin{tabular}{rl|c c c c c c}
                    & & $\alpha =24$ & $\alpha =20$ & $\alpha =16$ & $\alpha =12$ & $\alpha =8$ & $\alpha =4$ \\
\hline
\ref{def:filter1} &bijective          & $\surd$       & $\surd$      & $\surd$      & $\surd$     & $\surd$      & $\surd$      \\
\ref{def:filter2} & strictly increasing & $\surd$       & $\surd$      & $\surd$      & $\surd$     & $\surd$      & $\surd$      \\
\ref{def:filter3} &$\underset{\lambda}{\max} \{|\ph_\alpha(\kappa_\lambda,0)| \}$
                    & $0.021$        & $0.023$       & $0.017$       & $0.008$      & $0.021$       & $0.051$      \\
\ref{def:filter4} &$\sum_\la\norm{\ph_\alpha(\kappa_\la,\cdot)-\id}_{L^2}$       
                    & $0.31$       & $0.28$        & $0.22$       & $0.15$      & $0.09$      & $0.06$      \\
\midrule
\ref{ass:A1} &$\underset{\la,0\neq|x|<\alpha/\kappa}{\max} \left\{\frac{|\ph_\alpha(\kappa_\la, x)| \sqrt{\alpha}}{|x| \kappa_\la} \right\}$
                    & $28$        & $25$         & $22$        & $20$       & $21$       & $16$   \\
\ref{ass:A3} & $\underset{\la,x\neq 0}{\max} \left\{|\ph_\alpha(\ka_\la,x)| - 2|x|\right\}<0$  & $\surd$       & $\surd$      & $\surd$      & $\surd$     & $\surd$      & $\surd$      \\
\ref{ass:A2} &$\underset{\la,x\neq 0}{\min} \left\{\frac{|\ph_\alpha(\kappa_\lambda, x)|} {g_{\kappa_\lambda} (|x|)} \right\}$
                    & $1.31$        & $1.72$         & $1.82$        & $1.89$       & $1.58$       & $1.41$       
\end{tabular}}
\end{table*}

\section{Conclusion}
This paper investigated weakly convex frame-based image regularization. We established the connection between non-linear filtered diagonal frame decomposition (DFD) with learned filter functions and variational regularization with a weakly convex penalty term.

We presented experimental results by learning the so-called weakly convex regularizing filters, highlighted implementation specifics, and demonstrated the practical applicability of our method. The theoretical main section provided a comprehensive analysis of the stability and convergence of the filtered DFD with weakly convex regularizing filters, offering quantitative estimates to validate our approach.

We also conducted important technical work to improve understanding of the generalized proximity operator, separable sums of weakly convex functions, and their differentiability.

Finally, we concluded with a brief discussion of the consistency of our experimental results with the theoretical findings.

\section*{Acknowledgments}
This work was supported by the Austrian Science Fund (FWF) [grant number DOC 110].

\emergencystretch 2em
\printbibliography[heading=bibintoc]

@article{Goujon2024weaklyconvex,
author = {Goujon, Alexis and Neumayer, Sebastian and Unser, Michael},
title = {Learning Weakly Convex Regularizers for Convergent Image-Reconstruction Algorithms},
journal = {SIAM Journal on Imaging Sciences},
volume = {17},
number = {1},
pages = {91-115},
year = {2024}
}

@article{gribonval2020prox,
  TITLE = {{A characterization of proximity operators}},
  AUTHOR = {Gribonval, R{\'e}mi and Nikolova, Mila},
  JOURNAL = {{Journal of Mathematical Imaging and Vision}},
  PUBLISHER = {{Springer Verlag}},
  VOLUME = {62},
  PAGES = {773--789},
  YEAR = {2020},
  HAL_ID = {hal-01835101},
  HAL_VERSION = {v5},
}

@article{hauptmann2023convergent,
      title={Convergent regularization in inverse problems and linear plug-and-play denoisers}, 
      author={Andreas Hauptmann and Subhadip Mukherjee and Carola-Bibiane Schönlieb and Ferdia Sherry},
      year={2024},
        journal={Foundations of Computational Mathematics},
    publisher={Springer Science and Business Media LLC}
}

@article{burger2007error,
  title={Error estimation for Bregman iterations and inverse scale space methods in image restoration},
  author={Burger, Martin and Resmerita, Elena and He, Lin},
  journal={Computing},
  volume={81},
  pages={109--135},
  year={2007},
  publisher={Springer}
}

@article{obmann2023convergence,
      title={Convergence analysis of equilibrium methods for inverse problems}, 
      author={Daniel Obmann and Markus Haltmeier},
      eprint={2306.01421},
      archivePrefix={arXiv},
      year={2023}
}

@article{OBMANN2024,
title = {Convergence rates for critical point regularization},
journal = {Applied Mathematics and Computation},
volume = {471},
year = {2024},
issn = {0096-3003},
author = {Daniel Obmann and Markus Haltmeier}
}

@article{Grasmair2010,
year = {2010},
publisher = {},
volume = {26},
number = {11},
author = {Grasmair, Markus},
title = {Generalized Bregman distances and convergence rates for non-convex regularization methods},
journal = {Inverse Problems}
}

@article{Burger2004,
year = {2004},
publisher = {},
volume = {20},
number = {5},
author = {Martin Burger and  Stanley Osher},
title = {Convergence rates of convex variational regularization},
journal = {Inverse Problems}
}

@article{Ebner2023,
year = {2022},
publisher = {IOP Publishing},
volume = {39},
number = {1},
author = {Andrea Ebner and Markus Haltmeier},
title = {Convergence rates for the joint solution of inverse problems with compressed sensing data},
journal = {Inverse Problems}
}

@article{Abascal2018,
year = {2018},
publisher = {IOP Publishing},
volume = {34},
number = {12},
author = {J F P J Abascal and N Ducros and F Peyrin},
title = {Nonlinear material decomposition using a regularized iterative scheme based on the Bregman distance},
journal = {Inverse Problems}
}

@Incollection{Burger2016,
author="Burger, Martin",
editor="Hiriart-Urruty, Jean-Baptiste
and Korytowski, Adam
and Maurer, Helmut
and Szymkat, Maciej",
title="Bregman Distances in Inverse Problems and Partial Differential Equations",
bookTitle="Advances in Mathematical Modeling, Optimization and Optimal Control",
year="2016",
publisher="Springer International Publishing",
address="Cham",
pages="3--33",
isbn="978-3-319-30785-5",
}

@article{groetsch1984theory,
  title={The theory of {Tikhonov} regularization for {Fredholm} equations},
  author={Groetsch, Charles},
  journal={104p, Boston Pitman Publication},
  year={1984}
}

@article{Ga98,
author = {Gao, Ye},
journal = {Journal of Computational and Graphical Statistics},
number = {4},
pages = {469--488},
title = {{Wavelet shrinkage denoising using the non-negative garrote}},
volume = {7},
year = {1998}
}

@article{An21,
author = {Antoniadis, Anestis and Fan, Jianqing},
issn = {1537274X},
journal = {Journal of the American Statistical Association},
number = {455},
pages = {939--955},
title = {{Regularization of wavelet approximations}},
volume = {96},
year = {2001}
}

@inproceedings{quellmalz2023frame,
  title={A frame decomposition of the {F}unk-{R}adon transform},
  author={Quellmalz, Michael and Weissinger, Lukas and Hubmer, Simon and Erchinger, Paul},
  booktitle={International Conference on Scale Space and Variational Methods in Computer Vision},
  pages={42--54},
  year={2023},
  organization={Springer}
}

@article{goppel2023translation,
  title={Translation invariant diagonal frame decomposition of inverse problems and their regularization},
  author={G{\"o}ppel, Simon and Frikel, J{\"u}rgen and Haltmeier, Markus},
  journal={Inverse Problems},
  volume={39},
  number={6},
  year={2023},
  publisher={IOP Publishing}
}

@article{frikel2018efficient,
  title={Efficient regularization with wavelet sparsity constraints in photoacoustic tomography},
  author={Frikel, J{\"u}rgen and Haltmeier, Markus},
  journal={Inverse Problems},
  volume={34},
  number={2},
  year={2018},
  publisher={IOP Publishing}
}

@article{Hu21_2,
author = {Hubmer, Simon and Ramlau, Ronny},
issn = {13616420},
journal = {Inverse Problems},
number = {5},
pages = {1--32},
title = {{Frame decompositions of bounded linear operators in Hilbert spaces with applications in tomography}},
volume = {37},
year = {2021}
}

@article{Ca02,
author = {Cand{\`{e}}s, Emmanuel J. and Donoho, David L.},
issn = {00905364},
journal = {Annals of Statistics},
number = {3},
pages = {784--842},
title = {{Recovering edges in ill-posed inverse problems: Optimality of curvelet frames}},
volume = {30},
year = {2002}
}

@article{Eb23,
author = {Ebner, Andrea and Frikel, J{\"{u}}rgen and Lorenz, Dirk and Schwab, Johannes and Haltmeier, Markus},
journal = {Applied and Computational Harmonic Analysis},
title = {{Regularization of inverse problems by filtered diagonal frame decomposition}},
year = {2023}
}

@book{En96,
author = {Engl, Heinz W. and Hanke, Martin and Neubauer, Andreas},
booktitle = {Regularization of Inverse Problems},
isbn = {978-0-7923-4157-4},
publisher = {Springer},
title = {{Regularization of Inverse Problems}},
year = {1996}
}

@book{Ri03,
    author       = {Rieder, Andreas},
    year         = {2003},
    title        = {Keine Probleme mit Inversen Problemen : eine Einführung in ihre stabile Lösung},
    publisher    = {{Vieweg Verlag}},
    isbn         = {3-528-03198-0},
    pagetotal    = {300},
    language     = {german}
}

@article{Do95,
author = {Donoho, David L.},
journal = {Applied and Computational Harmonic Analysis},
pages = {101--126},
title = {{Wavelet-Vaguelet Decomposition}},
year = {1995}
}

@article{Be17, 
title={Modern regularization methods for inverse problems}, 
volume={27}, 
DOI={10.1017/S0962492918000016}, 
journal={Acta Numerica}, 
author={Benning, Martin and Burger, Martin}, 
year={2018}, 
pages={1–111}
}

@book{Ma09,
author = {Mallat, Stphane},
title = {A Wavelet Tour of Signal Processing, Third Edition: The Sparse Way},
year = {2008},
isbn = {0123743702},
publisher = {Academic Press, Inc.},
address = {USA},
edition = {3rd}
}

@article{Fr19,
author = {Frikel, J{\"{u}}rgen and Haltmeier, Markus},
issn = {2297024X},
journal = {Trends in Mathematics},
pages = {163--178},
title = {{Sparse regularization of inverse problems by operator-adapted frame thresholding}},
year = {2020}
}

@article{ebner2023convergence,
  title={Convergence of non-linear diagonal frame filtering for regularizing inverse problems},
  author={Ebner, Andrea and Haltmeier, Markus},
    eprint={2308.15666},
    archivePrefix={arXiv},
  year={2023}
}

@inproceedings{bauermeister2020learning,
  title={Learning spectral regularizations for linear inverse problems},
  author={Bauermeister, Hartmut and Burger, Martin and Moeller, Michael},
  booktitle={NeurIPS 2020 Workshop on Deep Learning and Inverse Problems},
  year={2020}
}

@article{kabri2022convergent,
  title={Convergent data-driven regularizations for ct reconstruction},
  author={Kabri, Samira and Auras, Alexander and Riccio, Danilo and Bauermeister, Hartmut and Benning, Martin and Moeller, Michael and Burger, Martin},
eprint={2212.07786},
    archivePrefix={arXiv},
  year={2022}
}

@article{soares2020sars,
  title={SARS-CoV-2 CT-scan dataset: A large dataset of real patients CT scans for SARS-CoV-2 identification},
  author={Soares, Eduardo and Angelov, Plamen and Biaso, Sarah and Froes, Michele Higa and Abe, Daniel Kanda},
  journal={MedRxiv},
  pages={2020--04},
  year={2020},
  publisher={Cold Spring Harbor Laboratory Press}
}

@article{donoho1995nonlinear,
  title={Nonlinear solution of linear inverse problems by wavelet--vaguelette decomposition},
  author={Donoho, David L},
  journal={Applied and computational harmonic analysis},
  volume={2},
  number={2},
  pages={101--126},
  year={1995},
  publisher={Elsevier}
}

\appendix
\section{Differentiability of Separable Functions}

\begin{lemma}\label{lem:diff_ineq}
Let $x, y \in \R$ and $t>t'>0$. Suppose $\ph \colon \R \rightarrow \R$ increasing with $\ph(0)=0$. Then, the following inequalities hold
\begin{align}
    \ph(x) y &\leq \frac{1}{t} \int_x^{x+t y} \ph(z) dz \leq \ph(x+ t y) y \label{lem:diff_ineq1}\\
    \ph(x) y &\leq \ph(x+t' y) y \leq \ph(x+t y) y \label{lem:diff_ineq2}
\end{align}
\end{lemma}

\begin{proof}
    Suppose $x, y \geq 0$, then $x \leq x +ty$  and since $\ph$ is increasing the right and left rule for Riemann sums show inequality \eqref{lem:diff_ineq1}. Inequality \eqref{lem:diff_ineq2} also follows by the monotonicity of $\ph$ and $y$ being positive.

    Now suppose $x \geq 0$ and $y \leq 0$, then $x+ty \leq x$. By convention we have $\int_x^{x+t y} \ph(z) dz = - \int_{x+t y}^x \ph(z) dz$ and even if $x+ty \leq 0$ the right and left rule for Riemann sums imply $-\ph(x+ty) ty \leq \int_{x+t y}^x \ph(z) dz \leq - \ph(x) t y$,  showing \eqref{lem:diff_ineq1}. Inequality \eqref{lem:diff_ineq2} follows by the monotonicity of $\ph$ and $y$ being negative.

    The cases $x,y \leq 0$ and $x\leq 0, y \geq 0$ are analogous. 
\end{proof}

\begin{lemma}[Differentiability of Convex Separable Functions]\label{lem:diff_sep}
Let $(r_\la)_{\la \in \La}$ be a family of convex and continuously differentiable functions $r_\la \colon \R \rightarrow \R$. Suppose $r_\la(0)=0$ and $r_\la'(0)=0$. Define $\Rc(x) \coloneqq \sum_\la r_\la(x_\la) \in [0, \infty]$ and $\Rc'(x) \coloneqq \sum_\la r'_\la(x_\la) \in [- \infty, \infty]$ for all $x \in \ell^2(\La)$. Recall, $\dom(\Rc) = \{x \in \ell^2(\La) \mid (r_\la(x_\la))_\la \in \ell^1 \}$ and $\dom(\Rc') = \{x \in \ell^2(\La) \mid (r'_\la(x_\la))_\la \in \ell^2 \}$. Then, $\Rc$ is Gateaux differentiable at $x \in \dom(\Rc)^\circ \cap \dom(\Rc')$ with gradient $\nabla \Rc(x)=(r'_\la(x_\la))_\la $.
\end{lemma}

\begin{proof}
By the fundamental theorem of calculus we have $r_\la(x)=\int_0^x r_\la'(z) dz$ for all $x \in \R$. 
Let $y \in \ell^2(\La)$ and $0 < t' < t$. Then, by Lemma \ref{lem:diff_ineq}, we have
\begin{align*}
     \sum_\la r_\la'(x_\la) y_\la  &\leq \sum_\la \frac{1}{t} \int_{x_\la}^{x_\la + t y_\la} r_\la'(z) dz \leq \sum_\la r_\la'(x_\la + t y_\la) y_\la \\
     \sum_\la r_\la'(x_\la) y_\la &\leq \sum_\la r_\la'(x_\la + t' y_\la) y_\la \leq \sum_\la r_\la'(x_\la + t y_\la) y_\la.
\end{align*}
By the monotone convergence theorem the limit $\lim_{t \downarrow 0} \sum_\la \frac{1}{t} \int_{x_\la}^{x_\la + t y_\la} r_\la'(z) dz$ exists and it holds
\begin{align*}
    \lim_{t \downarrow 0} \frac{\Rc(x+ty) - \Rc(x)}{t} = \lim_{t \downarrow 0} \sum_\la \frac{1}{t} \int_{x_\la}^{x_\la + t y_\la} r_\la'(z) dz = \sum_\la r_\la'(x_\la) y_\la = \inner{\nabla \Rc(x)}{y}.
\end{align*}
\end{proof}

\begin{proof}[Proof of Proposition \ref{prop:diff_R}]
Consider $r_\la(x) = (\ka_\la x)^2/2 + \qone_{\al,\la}(\ka_\la x)$. Then, $r_\la$ is continuously differentiable with $r_\la(0) = \qone_{\al,\la}(0) = 0$ and $r_\la'(0) = \ka_\la \qone_{\al,\la}'(0) = 0$. Furthermore, $r_\la'(x) = \ka_\la (\id + \qone_{\al,\la}')(\ka_\la x)$ is increasing by Proposition \ref{prop:prox} which implies that $r_\la$ is convex. Define $\Rc(x)=\sum_\la r_\la(x_\la)$ and note that $\dom(\Rc)^\circ \cap \dom(\Rc') = \dom(\Rc_\al)^\circ \cap \dom(\Rc_\al')$. By Lemma \ref{lem:diff_sep} we have that $\Rc$ is differentiable on $\dom(\Rc_\al)^\circ \cap \dom(\Rc_\al')$. In addition, note that Lemma \ref{lem:diff_sep} shows the differentiablility of $\norm{\Mo_\kao x}_2^2$ on $\ell^2(\La)$. Then,  $\Rc_\al$, as a sum of differentiable functions, is differentiable on $\dom(\Rc_\al)^\circ \cap \dom(\Rc_\al')$ and 
\[\nabla \Rc_\al (x) = \nabla \left( \Rc(x) - 1/2 \norm{\Mo_\kao x}_2^2 \right) = (\ka_\la \qone'_{\al,\la}(\ka_\la x_\la))_\la.\]
\end{proof}

\begin{lemma}\label{domR}
Let $(\ph_{\al})_{\al> 0}$ be a weakly convex regularizing filter and fix $\al>0$. 
Then for all $x = (x_\la)_{\la}  \in \ell^2(\La)$ the following identity  holds:
\[ \Rc_\al(x) = \norm{ \left( \int_0^{\ka_\la x_\la} \ph_{\al}(\ka_\la,\cdot)^{-1}(y)  dy \right)_\la }_1 - \frac{1}{2}\norm{\Mo_{\kao} x}_2^2.\]
In particular, $x \in \dom(\Rc_\al)$ if and only if $\left( \int_0^{\ka_\la x_\la} \ph_{\al}(\ka_\la,\cdot)^{-1}(y)  dy \right)_\la \in \ell^1(\La)$.
\end{lemma}
\begin{proof}
By Proposition \ref{prop:prox} we can calculate
\begin{align*}
\Rc_\al(x) & = \sum_{\la \in \La} \qone_{\al,\la}(\ka_\la x_\la) = \sum_{\la \in  \La} \int_0^{\ka_\la x_\la} \ph_{\al}(\ka_\la,\cdot)^{-1}(y) - y \, dy \\
& = \sum_{\la \in  \La} \int_0^{\ka_\la x_\la} \ph_{\al}(\ka_\la,\cdot)^{-1}(y)\, dy - \sum_{\la \in  \La} \frac{1}{2}(\ka_\la x_\la)^2 \\
& = \norm{ \int_0^{\ka_\la x_\la} \ph_{\al}(\ka_\la,\cdot)^{-1}(y) \, dy}_1 - \frac{1}{2} \norm{\Mo_{\kao} x}_2^2.
\end{align*}
\end{proof}

\section{Proof of Proposition~\ref{prop:conv}}

Choose $c, d > 0$ such that \ref{ass:A1} holds. Denote $z = (z_\la)_{\la \in \La} \in \ell^2(\La)$. 
Since $\sup_{\la \in \La} \ka_\la < \infty$ there are only finitely many $\la$ such that $\abs{z_\la} > c \al/\ka_\la$. 
Therefore we have
\begin{align*}
\norm{\Mo^+_{\kao} \circ \, \Phi_{\al,\kao}(z)}^2 &= \sum_{\la \in \La} \frac{\abs{\ph_{\al}(\ka_\la,z_\la)}^2}{\ka_\la^2} \\
&= \sum_{\abs{z_\la}\leq c \al/\ka_\la} \frac{\abs{\ph_{\al}(\ka_\la,z_\la)}^2}{\ka_\la^2}+ \sum_{\abs{z_\la} > c \al/\ka_\la} \frac{\abs{\ph_{\al}(\ka_\la,z_\la)}^2}{\ka_\la^2} \\
&\leq \sum_{\abs{z_\la}\leq c \al/\ka_\la} \frac{d^2}{\al}\abs{z_\la}^2 + C < \infty. 
\end{align*}
This shows that $\Mo^+_{\kao} \circ \, \Phi_{\al,\kao}(z) \in \ell^2(\La)$ for arbitrary $z \in \ell^2(\La)$.
Now, for the sequence $\norm{z-z^n} \leq \delta_n \rightarrow 0$ define the constant $a \coloneqq c \al/\sup_\la \ka_\la$ and 
the index sets $\La' \coloneqq \{\la \in \La \mid \abs{z_\la} > a/2 \}$, $\La^n \coloneqq \{\la \in \La \mid \abs{z^n_\la} > 3a/4 \}$ for $n \in \N$. Choose $N \in \N$ such that $\delta_n \leq a/4$ for all $n \geq N$. Then, for all $n \geq N$, we get $\La^n \subset \La'$ and 
\begin{align*}
\norm{\Mo^+_{\kao} \circ \, \Phi_{\al,\kao}(z^n)}^2 &= \sum_{\la \in \La^n} \abs{\frac{1}{\ka_\la}\ph_{\al}(\ka_\la,z^n_\la)}^2 + \sum_{\la \notin \La^n} \abs{\frac{1}{\ka_\la}\ph_{\al}(\ka_\la,z^n_\la)}^2 \\
& \leq \sum_{\la \in \La'} \abs{\frac{1}{\ka_\la}\ph_{\al}(\ka_\la,z^n_\la)}^2 + \sum_{\la \notin \La^n}\frac{d^2}{\al} \sabs{z^n_\la}^2 \\
& \leq C + \frac{d^2}{\al} (\norm{z}+\max_n \delta_n)^2, 
\end{align*}
where the first sum is bounded by the facts that $\La'$ is finite and $\ph_\al(\ka_\la, \cdot)$ is increasing together with the estimate $z_\la - \max\delta_n < z_\la^n < z_\la + \max \delta_n$.
Hence the sequence $(\Mo^+_{\kao} \circ \, \Phi_{\al,\kao}(z^n))_{n \in \N}$ is bounded.
Let $(e_\la)_{\la \in \La}$ be the unit basis. 
By the continuity of $\ph_{\al,\la}$ we have
\begin{align*}
\abs{\inner{\Mo^+_{\kao} \circ \, \Phi_{\al,\kao}(z^{n}) - \Mo^+_{\kao} \circ \, \Phi_{\al,\kao} (z)}{e_\la}} & = \abs{\frac{1}{\ka_\la} \ph_{\al}(\ka_\la,z_\la) - \frac{1}{\ka_\la} \ph_{\al}(\ka_\la,z^{n}_\la)} \\
& = \frac{1}{\ka_\la} \abs{\ph_{\al}(\ka_\la,z_\la) - \ph_{\al}(\ka_\la,z^{n}_\la)} \rightarrow 0
\end{align*}
Thus, $\Mo^+_{\kao} \circ \, \Phi_{\al,\kao}(z^n) \rightharpoonup \Mo^+_{\kao} \circ \, \Phi_{\al,\kao}(z)$.

Now, we show convergence. Assume $z\in \ran(\Mo_{\kao}) \cap \ran(\Phi_{\al,\kao})$ for all $\al \in (0,\tilde \al)$ and some $\tilde{\al}> 0$.
Define $x^n \coloneqq \Mo^+_{\kao} \circ \, \Phi_{\al_n,\kao}(z^n)$ and  $\La_n \coloneqq \{\la \in \La \mid \abs{z_\la^n} \geq c \al_n/\ka_\la \} $ for all $n \in \N$.
Then, we have for all $\la \in \La_n$ that $\norm{\Mo_{\kao}^+ z}_\infty + \delta_n/\ka_\la \geq \abs{z_\la}/\ka_\la + \delta_n/\ka_\la\geq \abs{z^n_\la}/\ka_\la \geq c \al_n/\ka^2_{\la} \geq C \delta_n^2/\ka_\la^2$ 
for a constant $C > 0$ (independent of $n$).
Therefore $\norm{\Mo_{\kao}^+ z}_\infty \geq  \delta_n/\ka_\la  (\delta_n/\ka_\la -1)$ which implies that the sequence $\left(\delta_n/\inf_{\la \in \La_n} \ka_\la \right)_{n \in \N}$ is bounded, otherwise one can create a contradiction to $z \in \ran(\Mo_{\kao})$.  
Therefore, for some constant $D > 0$,  
\begin{align*}
\norm{x^n}^2 & = \sum_{\la\in \La_n} \abs{\frac{1}{\ka_\la}\ph_{\al_n}(\ka_\la,z^n_\la)}^2 + \sum_{\underset{\ka_\la > \sqrt{\al_n}}{\la \notin \La_n}} \abs{\frac{1}{\ka_\la}\ph_{\al_n}(\ka_\la,z^n_\la)}^2 + \sum_{\underset{\ka_\la \leq \sqrt{\al_n}}{\la \notin \La_n}} \abs{\frac{1}{\ka_\la}\ph_{\al_n}(\ka_\la,z^n_\la)}^2 \\
& \leq \sum_{\la\in \La_n} \left(\frac{K \sabs{z^n_\la}}{\ka_\la}\right)^2 + \sum_{\underset{\ka_\la > \sqrt{\al_n}}{\la \notin \La_n}} \left(\frac{K \abs{z^n_\la}}{\ka_\la}\right)^2 + d \sum_{\underset{\ka_\la \leq \sqrt{\al_n}}{\la \notin \La_n}} \left(\frac{\abs{z^n_\la}}{\sqrt{\al_n}}\right)^2 \\
& \leq K^2 \sum_{\la\in \La_n} \left( \frac{\sabs{z_\la} + \abs{z_\la^n-z_\la}}{\ka_\la}\right)^2 + (K^2 + d) \sum_{{\la \notin \La_n}} \left(\frac{\abs{z_\la}}{\ka_\la} + \frac{\abs{z_\la^n-z_\la}}{\sqrt{\al_n}}\right)^2 \\
& \leq K^2 \left(\norm{\Mo_{\kao}^+ z} + \frac{\delta_n}{\inf_{\la \in \La_n} \ka_\la} \right)^2 + (K^2 + d) \left(\norm{\Mo_{\kao}^+ z} +\frac{\delta_n}{\sqrt{\al_n}}  \right)^2 \\
& \leq D.
\end{align*}
Thus $(x^n)_n$ is bounded and has a weakly convergent subsequence. 
Let $(x^{n_l})_l$ be such a weakly convergent subsequence with weak limit $x^+$. 
Now we show for a fixed $\la \in \La$ that $\qone_{\al_{n_l},\la}(z_\la) \rightarrow 0$ and $\qone_{\al_{n_l},\la}(\ka_\la x^{n_l}_\la) \rightarrow 0$ as $l \rightarrow \infty$.
Recall that by Proposition \ref{prop:prox} we have the representation 
$\qone_{\al,\la}(x) = \int_0^{x} \ph_{{\al}}(\ka_\la,\cdot)^{-1}(y) -y \, dy$.
By Assumption \ref{def:filter4} it holds $\lim_{\al \rightarrow 0} \ph_{\al}(\ka_\la,x) = x$ and by bijectivity it holds also $\lim_{\al \rightarrow 0} \ph_{\al}(\ka_\la,\cdot)^{-1}(y) = y$.
Furthermore, from Assumption \ref{ass:A2} there exists $g \colon [0,\infty) \rightarrow [0,\infty)$ bicontinuous such that
$\forall \al \in (0,\tilde \al) \colon \abs{\ph_{\al}(\ka_\la,y)} \geq g(\abs{y})$
and therefore $\abs{\ph_{\al}(\ka_\la,\cdot)^{-1}(y)} \leq g^{-1}(\abs{y})$. 
By the dominated convergence holds $\qone_{\al,\la}(z_\la) \rightarrow 0$ as $\al \rightarrow 0$. Further, we have
\begin{align*}
\abs{\qone_{\al_{n_l},\la} (\ka_\la x^{n_l}_\la)} &\leq \int_0^{\ka_\la \abs{x^{n_l}_\la}} \abs{\ph_{{\al_{n_l}}}(\ka_\la,\cdot)^{-1}(\sign(x^{n_l}_\la)y) - \sign(x^{n_l}_\la)y} \, dy \\
& \leq \int_0^{\ka_\la \underset{l \in \N}{\max}\abs{x^{n_l}_\la}} \abs{\ph_{{\al_{n_l}}}(\ka_\la,\cdot)^{-1}(\sign(x^{n_l}_\la)y) - \sign(x^{n_l}_\la)y} \, dy
\end{align*}
and therefore also by the dominated convergence $\qone_{\al_{n_l},\la} (\ka_\la x^{n_l}_\la) \rightarrow 0$ as $\al \rightarrow 0$.

By Lemma \ref{lem_main} , we have $x^{n_l}_\la = \argmin\{\frac{1}{2}\abs{\ka_\la x - z^{n_l}_\la}^2 + \qone_{\al_{n_l},\la}(\ka_\la x) \}$, thus
\begin{align*}
\frac{1}{2}\abs{\ka_\la x^{n_l}_\la - z^{n_l}_\la}^2 + \qone_{\al_{n_l},\la}(\ka_\la x^{n_l}_\la)
& \leq \frac{1}{2}\abs{z_\la - z^{n_l}_\la}^2 + \qone_{\al_{n_l},\la}(z_\la) \\ 
& \leq \frac{1}{2} \delta_{n_l}^2 + \qone_{\al_{n_l},\la}(z_\la)
\rightarrow 0, 
\end{align*}
as $l \rightarrow \infty$. 
Since $\qone_{\al_{n_l},\la} (\ka_\la x^{n_l}_\la) \rightarrow 0$ it follows that $\abs{\ka_\la x^{n_l}_\la - z^{n_l}_\la} \rightarrow 0$.  
Therefore, $x^+_\la = z_\la/\ka_\la$ for every $\la \in \La$. 
This holds for every weakly convergent subsequence, so we can conclude that $x^n \rightharpoonup \Mo_{\kao}^+ z$.

\end{document}